\documentclass[11pt,a4paper,twoside]{article}

\usepackage{titling}
\usepackage{titletoc}
\usepackage{url}

\usepackage[utf8]{inputenc}
\usepackage[T1]{fontenc}
\usepackage{amsmath}
\usepackage{amsfonts}
\usepackage{amssymb}
\usepackage{amsthm}
\usepackage{verbatim}

\usepackage{mathrsfs}
\usepackage{mathtools}
\usepackage{stmaryrd}
\usepackage[all]{xy}
\usepackage{makeidx}
\usepackage{natbib}
\bibpunct{(}{)}{,}{a}{}{,}
\usepackage[french,english]{babel}
\usepackage{xspace}
\usepackage{float}

\usepackage{etoolbox}

\makeatletter\def\th@plain{\slshape}\makeatother
\makeatletter\patchcmd{\th@remark}{\itshape}{\slshape}{}{}\makeatother

\newcounter{bidon}
\newcommand{\rdb}{\refstepcounter{bidon}}

\newcommand \sibrouillon[1]{}

\usepackage[bookmarksopen=false,breaklinks=true,%
      backref=page,pagebackref=true,plainpages=false,%
      hyperindex=true,pdfstartview=FitH,%
      pdfpagelabels=true,
      colorlinks=true,linkcolor=blue,%
      citecolor=red,urlcolor=red,
      colorlinks=true%
      ]%
   {hyperref}

\begin{document}

\newcommand \eti{^\times }

\newcommand \vu {\vee} 
\newcommand \vi {\wedge} 
\newcommand \Vu {\bigvee\nolimits}
\newcommand \Vi {\bigwedge\nolimits}

\newcommand \RR{\mathbb{R}}
\newcommand \QQ{\mathbb{Q}}
\newcommand \NN{\mathbb{N}}

\newcommand \alb  {\allowbreak}
\newcommand \uX  {\underline{X}}
\newcommand \fm  {\mathfrak{m}}

\renewcommand \O {\mathrm{O}}
\newcommand \e {\mathrm{e}}
\newcommand \f {\mathrm{f}}
\renewcommand \mod {\,\mathrm{mod}\,}

\newcommand \piv {\mathrm{piv}}
\newcommand \cq {\mathrm{cpiv}}
\newcommand \Frac {\mathrm{Frac}}
\newcommand \idx {\mathrm{index}}
\renewcommand \Im {\mathrm{Im}}
\newcommand \PrimCoeff {\mathrm{PrimCoeff}}
\newcommand \PrimExp {\mathrm{PrimExp}}
\newcommand \PrimMon {\mathrm{PrimMon}}
\newcommand \Rad {\mathrm{Rad}}
\newcommand \Res {\mathrm{Res}}
\newcommand \Sat {\mathrm{Sat}}
\newcommand \Tr {\mathrm{Tr}}

\newcommand \som {\sum\nolimits}

\newcommand \gen[1]{\langle{#1}\rangle}
\newcommand \tra[1]{{\,^\mathrm{t}#1}}
\newcommand \eqg[1]{\buildrel{#1}\over \simeq}
\newcommand\so[1]{\left\{{#1}\right\}}
\newcommand\sotq[2]{\so{\,#1\mid#2\,}}

\newcommand \RedPrim{\mathsf{RedPrim}}

\newcommand \V{\mathbf{V}}
\newcommand \A{\mathbf{A}}
\newcommand \B{\mathbf{B}}
\newcommand \C{\mathbf{C}}
\renewcommand \k{\mathbf{k}}
\newcommand \K{\mathbf{K}}
\newcommand \R{\mathbf{R}}
\newcommand \T{\mathbf{T}}
\newcommand \gA{\A}
\newcommand \gB{\B}
\newcommand \gC{\C}
\newcommand \gk{\k}
\newcommand \gK{\K}
\newcommand \gR{\R}
\newcommand \gT{\T}
\newcommand \gV{\V}
\newcommand \kX{\k[X]}
\newcommand \KX{\K[X]}
\newcommand \RX{\R[X]}
\newcommand \VX{\V[X]}

\newcommand \cI{\mathcal{I}}
\newcommand \cM{\mathcal{M}}
\newcommand \cN{\mathcal{N}}
\newcommand \cS{\mathcal{S}}

\newcommand \noi {\noindent}
\renewcommand \ss {\smallskip}
\newcommand \sni {\smallskip\noindent}
\newcommand \ms {\medskip}
\newcommand \mni {\medskip\noindent}
\newcommand \bs {\bigskip}
\newcommand \bni {\bigskip\noindent}

\newcommand \ov[1] {\overline{#1}}

\newcommand \intervalle[1] {[ \negthinspace [#1] \negthinspace ]}

\newcommand\hsz{\\ }
\newcommand\hsu{\\ \hspace*{4mm}}
\newcommand\hsd{\\ \hspace*{8mm}}
\newcommand\hst{\\ \hspace*{1,2cm}}
\newcommand\hsq{\\ \hspace*{1,6cm}}
\newcommand\hsc{\\ \hspace*{2cm}}
\newcommand\hsix{\\ \hspace*{2,4cm}}
\newcommand\hsept{\\ \hspace*{2,8cm}}

\thispagestyle{empty}
~ 

\noindent The French version of this paper appeared in\\  \emph{Communications in Algebra}, {\bf 42}: 3768–3781, 2014 
\\
DOI: 10.1080/00927872.2013.794360

\bigskip \noindent In this file you find the English version starting on the page  numbered E1.

\medskip \noindent  {\large \bf An algorithm for computing syzygies on $\V[X]$ when $\V$ is a valuation domain}

\bigskip \noindent  
Then the French version begins on the page numbered F1.

\medskip\noindent   {\large \bf Un algorithme pour le calcul des syzygies sur  $\V[X]$ dans le cas où $\V$ est un domaine de valuation} 

\smallskip \noindent Le lecteur ou la lectrice sera sans doute surprise de l'alternance des sexes ainsi que de l'orthographe du mot 'corolaire', avec d'autres innovations auxquelles elle n'est pas habituée. En fait, nous avons essayé de suivre au plus près les préconisations de l'orthographe nouvelle recommandée, telle qu'elle est enseignée aujourd'hui dans les écoles en France.

\bigskip\noindent   {\large \bf Authors}

\smallskip 
\noindent Henri Lombardi

\noindent Laboratoire de Mathématiques de Besançon, CNRS UMR 6623, 
Université Bourgogne Franche-Comté, 25030 BESANCON cedex, FRANCE,\\
email: henri.lombardi@univ-fcomte.fr, \\
url: \url{http://hlombardi.free.fr}

\smallskip \noindent 
Claude Quitté

\noindent Laboratoire de Mathématiques,
SP2MI, Boulevard 3, Teleport 2, BP 179,
86960 FUTUROSCOPE Cedex, FRANCE,\\
email: quitte@math.univ-poitiers.fr

\smallskip \noindent Ihsen Yengui

\noindent 
Département de Mathématiques,  Faculté des Sciences de Sfax, 3000 Sfax,
Tunisia,\\ email: ihsen.yengui@fss.rnu.tn

\bigskip The results in this paper have been improved in \citealt*{DMY2016} 
(giving precise bounds) and generalised
for the rings $\bold V[X_1,\dots,X_k]$ in \citealt*{DVY2015}.
\begin{abstract} 
We give an algorithm for computing the $\bold V$-saturation of any finitely-generated submodule of $\bold V[X_1,\dots,X_k]^m$ $(k\in \NN, m \in \NN^*)$, where $\bold V$ is a valuation domain. Our algorithm is based on a notion of “echelon form” which ensures its correctness. The proposed algorithm terminates when two (Hilbert) series on the quotient field and the residue field of $\bold V$ coincide. As application, our algorithm computes syzygies over $\bold V[X_1,\dots,X_k]$.

\end{abstract}

{\small 
\bibliographystyle{plainnat}
%
%
%

}

\normalsize

\clearpage
\newpage
\thispagestyle{empty}
~

\pagestyle{headings}
\patchcmd{\sectionmark}{\MakeUppercase}{}{}{}
\setcounter{page}{0}\renewcommand\thepage{E\arabic{page}}

\begingroup

\newcommand\tsf[1]{\textbf{\textsf{#1}}}
\newcommand\Vrai{\mathsf{True}}
\newcommand\Faux{\mathsf{False}}
\newcommand\ET{\mathsf{ and }}
\newcommand\OU{\mathsf{ or }}

\newcommand\sialors [1]{\textbf{if } $#1$ \textbf{ then }}

\newcommand\pour[3]{\textbf{for } $#1$ \textbf{ from } $#2$
      \textbf{ to} $#3$ \textbf{do  }}
\newcommand\tantque[1]{\textbf{while} $#1$ \textbf{ do }}
\newcommand\finpour{\textbf{end for}}
\newcommand\fintantque{\textbf{end while }}
\newcommand\Debut{\\[1mm] \textbf{Start }}
\newcommand\fin{\textbf{\\ End.}}
\newcommand\Entree{\\[1mm] \textbf{Input: }}
\newcommand\Sortie{\\ \textbf{Output : }}
\newcommand\Varloc{\\ \textbf{Local variables : }}

\newcommand \srl{ linearly recurrent sequence\xspace}

\newcommand\comm{\rdb
\noi{\it Comment. }} 

\newcommand\COM[1]{\rdb
\noi{\it Comment #1. }}

\newcommand\comms{\rdb
\noi{\it Comments. }}

\newcommand\Pb{\rdb
\noi{\bf Problem. }}

\newcommand \rem{\rdb
\noi{\sl Remark. }}

\newcommand \REM[1]{\rdb
\noi{\sl Remark#1. }}

\newcommand \rems{\rdb
\noi{\sl Remarks. }}

\newcommand \exl{\rdb
\noi{\bf Example. }}

\newcommand \EXL[1]{\rdb
\noi{\bf Example: #1. }}

\newcommand \exls{\rdb
\noi{\bf Examples. }}

\newcommand\gui[1]{``{#1}''}

\newcommand \thref[1] {Theorem~\ref{#1}}
\newcommand \paref[1] {page~\pageref{#1}}
\newcommand \pstfref[1] {Positivstellensatz formel~\ref{#1}}
\newcommand \pstref[1] {Positivstellensatz~\ref{#1}}

\newcommand\subsubsec[1] {\subsubsection*{#1}}

\newcommand \hdr {induction hypothesis\xspace}
\newcommand \ssi {if and only if\xspace}
\newcommand \cnes {necessary and sufficient condition\xspace}
\newcommand \spdg {without loss of generality\xspace}
\newcommand \Propeq {T.F.A.E.\xspace}
\newcommand \propeq {t.f.a.e.\xspace}
\newcommand \disept {17$^{th}$ Hilbert's problem\xspace}

\def \cad {\textit{i.e.}\ }


\newcommand \Amo {$\gA$-module\xspace}
\newcommand \Amos {$\gA$-modules\xspace}

\newcommand \Bmo {$\gB$-module\xspace}
\newcommand \Bmos {$\gB$-modules\xspace}

\newcommand \Cmo {$\gC$-module\xspace}
\newcommand \Cmos {$\gC$-modules\xspace}

\newcommand \kmo {$\gk$-module\xspace}
\newcommand \kmos {$\gk$-modules\xspace}

\newcommand \Kmo {$\gK$-module\xspace}
\newcommand \Kmos {$\gK$-modules\xspace}

\newcommand \Lmo {$\gL$-module\xspace}
\newcommand \Lmos {$\gL$-modules\xspace}

\newcommand \Rmo {$\gR$-module\xspace}
\newcommand \Rmos {$\gR$-modules\xspace}

\newcommand \Vmo  {$\gV$-module\xspace}
\newcommand \Vmos {$\gV$-modules\xspace}

\newcommand \VXmo  {$\VX$-module\xspace}

\newcommand \Zmo {$\gZ$-module\xspace}
\newcommand \Zmos {$\gZ$-modules\xspace}

\newcommand \ZZmo {$\ZZ$-module\xspace}
\newcommand \ZZmos {$\ZZ$-modules\xspace}

\newcommand \Ali {$\gA$-\ali}
\newcommand \Alis {$\gA$-\alis}

\newcommand \Alg {$\gA$-\alg}
\newcommand \Algs {$\gA$-\algs}

\newcommand \kev {$\gk$-vector space\xspace}
\newcommand \kevs {$\gk$-vector spaces\xspace}

\newcommand \Kev {$\gK$-vector space\xspace}
\newcommand \Kevs {$\gK$-vector spaces\xspace}

\newcommand \klg {$\gk$-\alg}
\newcommand \klgs {$\gk$-\algs}

\newcommand \Klg {$\gK$-\alg}
\newcommand \Klgs {$\gK$-\algs}


\newcommand \agq {algebraic\xspace}

\newcommand \alg {algebra\xspace}
\newcommand \algs {algebras\xspace}

\newcommand \agB {Boolean \alg}

\newcommand \algo{algorithm\xspace}
\newcommand \algos{algorithms\xspace}

\newcommand \algq{algorithmic\xspace}

\newcommand \ali {\lin map\xspace}
\newcommand \alis {\lin maps\xspace}

\newcommand \anar {\ari \ri}
\newcommand \anars {\ari \ris}
\newcommand \Anars {\Ari \ris}

\newcommand \ari{arith\-metic\xspace}

\newcommand \auto {automorphism\xspace}
\newcommand \autos {automorphisms\xspace}


\newcommand \cac {algebraically closed field\xspace}
\newcommand \cacs {algebraically closed fields\xspace}

\newcommand \carn{characterization\xspace}  
\newcommand \carns{characterizations\xspace}  

\newcommand \cdi{discrete field\xspace}  
\newcommand \cdis{discrete fields\xspace}  

\newcommand \coe {coefficient\xspace}
\newcommand \coes {coefficients\xspace}

\newcommand \cohc {coherence\xspace}

\newcommand \coh {coherent\xspace}
\newcommand \cohs {\coh}

\newcommand \coli {linear combination\xspace}
\newcommand \colis {linear combinations\xspace}

\newcommand \com {comaximal\xspace}

\newcommand \coo {coordinate\xspace}
\newcommand \coos {coordinates\xspace}

\newcommand \dcd {résidually discrete\xspace}
\newcommand \dcds {\dcd}

\newcommand \ddp {Pr\"ufer domain\xspace}
\newcommand \ddps {Pr\"ufer domains\xspace}

\newcommand \ddv {valuation domain\xspace}
\newcommand \ddvs {valuation domains\xspace}

\newcommand \Demo{Proof\xspace}     

\newcommand \demo{proof\xspace}     
\newcommand \demos{proofs\xspace}     

\newcommand \dem{\demo}
\newcommand \dems{\demos}

\newcommand \ddk {Krull dimension\xspace}

\newcommand \dfn{definition\xspace}  
\newcommand \dfns{definitions\xspace}  

\newcommand \discri{discriminant\xspace}
\newcommand \discris{discriminants\xspace}

\newcommand \dok {Dedekind domain\xspace}
\newcommand \doks {Dedekind domains\xspace}

\newcommand \dve {divisibility\xspace}

\newcommand \dvz {zero divisor\xspace}
\newcommand \dvzs {zero divisors\xspace}

\newcommand \eco{\com \elts}  

\newcommand \egmt {also\xspace}

\newcommand \egt{equality\xspace} 
\newcommand \egts{equalities\xspace} 

\newcommand \elr{elementary\xspace}  
\newcommand \elrs{\elr}  

\newcommand \elt{element\xspace}  
\newcommand \elts{elements\xspace}  

\def \endo {endomorphism\xspace}
\def \endos {endomorphisms\xspace}

\newcommand \entrel {entailment relation\xspace}
\newcommand \entrels {entailment relations\xspace}

\newcommand \eqv  {equivalent\xspace}

\newcommand \evc{vector space\xspace} 
\newcommand \evcs{vector spaces\xspace} 


\newcommand \fab {bounded \fcn}
\newcommand \fabs {bounded \fcns}

\newcommand \fac {total \fcn}
\newcommand \facz {total \fcnz}

\newcommand \fap {partial \fcn}
\newcommand \faps {partial \fcns}

\newcommand \fcn {factorisation\xspace}
\newcommand \fcns {factorisations\xspace}

\newcommand \fdi{strongly discrete\xspace} 


\newcommand\gmq{geometric\xspace}

\newcommand\gne{generalized\xspace}

\newcommand\gnl{general\xspace}

\newcommand\gnlt{generally\xspace}

\newcommand\gnn{generalisation\xspace}
\newcommand\gnns{generalisations\xspace}

\newcommand\gnq{generic\xspace}

\newcommand\grl{$\ell$-group\xspace}
\newcommand\grls{$\ell$-groups\xspace}

\newcommand \gtr{generator\xspace}  
\newcommand \gtrs{generators\xspace}  


\newcommand \homo {homomorphism\xspace}
\newcommand \homos {homomorphisms\xspace}

\newcommand \id {ideal\xspace}
\newcommand \ids {ideals\xspace}

\newcommand \idd {de\-ter\-mi\-nantal \id}
\newcommand \idds {de\-ter\-mi\-nantal \ids}
\newcommand \iddz {de\-ter\-mi\-nantal \idz}
\newcommand \iddsz {de\-ter\-mi\-nantal \idsz}

\newcommand \idema {maximal \id}
\newcommand \idemas {maximal \ids}

\newcommand \idep {prime \id}
\newcommand \ideps {prime \ids}

\newcommand \idemi {minimal prime\xspace}
\newcommand \idemis {minimal primes\xspace}

\newcommand \idf {Fitting \id}
\newcommand \idfs {Fitting \ids}

\newcommand \idm {idempotent\xspace}
\newcommand \idms {idempotents\xspace}

\newcommand \idtr {indeterminate\xspace}
\newcommand \idtrs {indeterminates\xspace}

\newcommand \ifr {fractional \id}
\newcommand \ifrs {fractional \ids}

\newcommand \itf {\tf \id}
\newcommand \itfs {\tf \ids}

\newcommand \iso {isomorphism\xspace}
\newcommand \isos {isomorphisms\xspace}

\newcommand \iv {invertible\xspace}

\newcommand \lec {reader\xspace}

\newcommand \lgb {local global\xspace}

\newcommand \lin {linear\xspace}

\newcommand \lon {localisation\xspace}
\newcommand \lons {localisations\xspace}

\newcommand \lop {\lot principal\xspace}

\newcommand \losd {\lot \sdz\xspace}

\def \lot {locally\xspace}

\newcommand \mlp {principal \lon matrix\xspace}
\newcommand \mlps {principal \lon matrices\xspace}

\newcommand \mnp {manipulation\xspace}
\newcommand \mnps {manipulations\xspace}
\newcommand \mnr {\elr \mnp}
\newcommand \mnrs {\elr \mnps}

\newcommand \mo {monoid\xspace}
\newcommand \mos {monoids\xspace}
\newcommand \moco {\com \mos}

\newcommand \mpf {\pf module\xspace}
\newcommand \mpfs {\pf modules\xspace}

\newcommand \mpn {\pn matrix\xspace}
\newcommand \mpns {\pn matrices\xspace}

\newcommand \mpr {\pro module\xspace}
\newcommand \mprs {\pro modules\xspace}

\newcommand \mprn {\prn matrix\xspace}
\newcommand \mprns {\prn matrices\xspace}

\newcommand \mptf {\ptf module\xspace}
\newcommand \mptfs {\ptf modules\xspace}

\newcommand \mrc {projective module of constant rank\xspace}
\newcommand \mrcs {projective modules of constant rank\xspace}

\newcommand \mtf {\tf module\xspace}
\newcommand \mtfs {\tf modules\xspace}


\newcommand \ncr{necessary\xspace}

\newcommand \ncrt{necessarily\xspace}

\newcommand \ndz {regular\xspace}

\newcommand \noe {Noetherian\xspace}
\newcommand \noes {\noe}
\newcommand \noee {\noe}
\newcommand \noees {\noe}

\newcommand \noet {Noetherianity\xspace}

\newcommand \noco {\noe\coh}

\newcommand \nst {Nullstellensatz\xspace}
\newcommand \nsts {Nullstellensätze\xspace}

\newcommand \odz {Zariski open set\xspace}

\newcommand \oqc {\qc open set\xspace}
\newcommand \oqcs {\qc open sets\xspace}


\newcommand \pa {saturated pair\xspace}
\newcommand \pas {saturated pairs\xspace}

\newcommand \pb{problem\xspace}  
\newcommand \pbs{problems\xspace}

\newcommand \peq {purely equational\xspace}

\newcommand \pf {finitely presented\xspace}

\newcommand \plg {\lgb principle\xspace}
\newcommand \plgs {\lgb principles\xspace}

\newcommand \pn {presentation\xspace}
\newcommand \pns {presentations\xspace}

\newcommand \pol {polynomial\xspace}
\newcommand \pols {polynomials\xspace}

\newcommand \polcar {characteristic \pol}

\newcommand \prc {rank constant \pro}

\newcommand \prmt {précisely\xspace}

\newcommand \prn {projection\xspace}
\newcommand \prns {projections\xspace}

\newcommand \pro {projective\xspace}

\newcommand \proi {potential prime\xspace}
\newcommand \prois {potential primes\xspace}

\newcommand \proc {potential chain\xspace}
\newcommand \procs {potential chains\xspace}

\newcommand \proel {elementary \proc}
\newcommand \proels {elementary \procs}
\newcommand \proelo {\proel of length }
\newcommand \proelos {\proels of length }

\newcommand \prolo {\proc of length }
\newcommand \prolos {\procs of length }

\newcommand \prt {property\xspace}
\newcommand \prts {properties\xspace}

\newcommand \pst {Positivstellensatz\xspace}
\newcommand \psts {Positivstellens\"atze\xspace}

\newcommand \ptf {\tf \pro}


\newcommand \qc {quasi-compact\xspace}

\newcommand \qi {quasi integral\xspace}

\newcommand \rcf {real closed field\xspace}
\newcommand \rcfs {real closed fields\xspace}

\newcommand \rdl {linear dependance relation\xspace}
\newcommand \rdls {linear dependance relations\xspace}

\newcommand \rdi {integral dependance relation\xspace}
\newcommand \rdis {integral dependance relations\xspace}

\newcommand \rdt {residually\xspace}

\newcommand \recu {induction\xspace}
\newcommand \recus {inductions\xspace}

\newcommand \ri {ring\xspace}
\newcommand \ris {rings\xspace}


\newcommand \sad {dynamical algebraic structure\xspace}
\newcommand \sads {dynamical algebraic structures\xspace}

\newcommand \sdz {without \dvz}
\newcommand \sdzz {without \dvzz}

\newcommand \sgr {\gtr set\xspace}
\newcommand \sgrs {\gtr sets\xspace}

\newcommand \sli {\lin \sys}
\newcommand \slis {\lin \syss}

\newcommand \sys {system\xspace}
\newcommand \syss {systems\xspace}

\newcommand \talg {Horn theory\xspace}
\newcommand \talgs {Horn theories\xspace}

\newcommand \tco {coherent theory\xspace}
\newcommand \tcos {coherent theories\xspace}

\newcommand \tdy {dynamical theory\xspace}
\newcommand \tdys {dynamical theories\xspace}

\newcommand \tel {regular theory\xspace}
\newcommand \tels {regular theories\xspace}

\newcommand \telri {cartesian theory\xspace}
\newcommand \telris {cartesian theories\xspace}

\newcommand \tf {finitely generated\xspace}

\newcommand \tfo {formal theory\xspace}
\newcommand \tfos {theory formelles\xspace}

\newcommand \tgm {\gmq theory\xspace}
\newcommand \tgms {\gmq theories\xspace}

\newcommand \Tho {Theorem\xspace}
\newcommand \tho {theorem\xspace}
\newcommand \thos {theorems\xspace}

\newcommand \tpe {purely equational theory\xspace}

\newcommand \trdi {distributive lattice\xspace}
\newcommand \trdis {distributive lattices\xspace}

\newcommand \vfn {verification\xspace}
\newcommand \vfns {verifications\xspace}

\newcommand \zed {zero-dimensional\xspace}

\newcommand \zedr {reduced \zed}
\newcommand \zedrs {\zedr}


\newcommand \cov {constructive\xspace}

\newcommand \coma {\cov \maths}
\newcommand \clama {classical \maths}

\renewcommand \cot {constructively\xspace}

\newcommand \mathe {mathematical\xspace}
\newcommand \maths {mathematics\xspace}

\newcommand \matn {mathematician\xspace}

\newcommand \pte {excluded middle principle\xspace}

\newcommand \prco {\cov proof\xspace}
\newcommand \prcos {constructive proofs\xspace}

\newcommand \tcg {compactness theorem\xspace}
\newcommand \Tcgi {The \tcg implies the following result. }
%


\theoremstyle{plain}
\newtheorem{theorem}{Theorem}[section]
\newtheorem{thdef}[theorem]{Theorem and definition}
\newtheorem{lemma}[theorem]{Lemma}
\newtheorem{corollary}[theorem]{Corollary}
\newtheorem{proposition}[theorem]{Proposition}
\newtheorem{propdef}[theorem]{Proposition and definition}
\newtheorem{plcc}[theorem]{Concrete local-global principle}
\newtheorem{fact}[theorem]{Fact}
\newtheorem{algorithm}{Algorithm}

\theoremstyle{definition}
\newtheorem{conjecture}[theorem]{Conjecture}
\newtheorem{definition}[theorem]{Definition}
\newtheorem{definitions}[theorem]{Definitions}
\newtheorem{notation}[theorem]{Notation}
\newtheorem{definota}[theorem]{Definition and notation} 
\newtheorem{convention}[theorem]{Convention}
\newtheorem{context}[theorem]{Context}
\newtheorem{problem}[theorem]{Problem}
\newtheorem{question}[theorem]{Question}

\theoremstyle{remark}
\newtheorem{remark}[theorem]{Remark}
\newtheorem{remarks}[theorem]{Remarks}
\newtheorem{commente}[theorem]{Comment}
\newtheorem{comments}[theorem]{Comments}
\newtheorem{example}[theorem]{Example}
\newtheorem{examples}[theorem]{Examples}

\def\proofname{\textsl{Proof}}

\title{An algorithm for computing syzygies on $\mathbf{V}[X]$\\
when $\mathbf{V}$ is a valuation domain}
\author{
Henri Lombardi,
Claude Quitté \& Ihsen Yengui
}

\date{2014, translation 2023}

\def\thefootnote{\arabic{footnote}}

\startcontents[english]

\newcommand\hum[1]{}

\maketitle

\selectlanguage{english}
\begin{abstract}
\smallskip
We give an algorithm for computing the $\V$-saturation of any finitely
generated submodule of $\V[X]^n$ ($n \in \mathbb{N}^*$),
where $\V$ is a valuation domain. This allows us to compute a finite
system of generators for the syzygy module of any finitely generated
submodule of $\V[X]^k$.
\end{abstract}

\sni {\small\textbf{Key words:}  Saturation, Coherence, Syzygies, Valuation domains, Computer algebra, Constructive Algebra.}

\setcounter{tocdepth}{4}
\markboth{Contents}{Contents}
\small

\printcontents[english]{}{1}{}
\normalsize

\pagestyle{myheadings} \markboth{An algorithm for computing syzygies on $\mathbf{V}[X]$}{H. Lombardi, C. Quitté and I. Yengui}


\section*{Introduction} \label{sec Introduction}
\addcontentsline{toc}{section}{Introduction}

It is folklore (see e.g. \cite[Glaz, Th. 7.3.3]{Glaz}) that for a valuation domain~$\V$, the ring $\V[X]$ is coherent, (i.e., the syzygy module of a \itf of $\VX$ is \tf).
The proof given in \cite{Glaz} uses a difficult result in a deep paper  \cite{GR}.
There is nevertheless no known  general  algorithm  for this remarkable result.

For a \noe \coh ring $\R$ (not \ncrt a \ddv),
it is known that the \pol \ri $\R[X_1,\dots,X_n]$ is also \noe \coh.
A \prco is given in \cite{ric74}
and carefully explained in the book \citealt*{MRR}.
In this case it is also possible to use Gröbner bases which were introduced by  Buchberger for \pol rings over fields
(see e.g.\ \cite{Lou,HY,Y}).

Nevertheless,  \noet is not really used in the case of fields since the coherence result
is also easily proved for \zedr rings
(also called Von Neumann regular rings, or absolutely flat rings).

In \citealt*{LSY}, an \algo is given for the computation of a Gröbner basis for a \itf of $\VX$ if the \ddv is of dimension~1
(not \ncrt \noe).
From this  follows an \algo for the \cohc of a \itf of~$\VX$.

Let us recall (see e.g.\ \citealt*{MRR}) that if a ring is \coh then any \mpf $M$ is a \coh module
(i.e., the syzygy module of any \tf sub-module of $M$ is \tf).

In this paper, we give an \algo with no \noe hypothesis, and no hypothesis about the Krull dimension,
for the computation of a finite generating set for the syzygy module of a finite family of vectors in $\VX^k$.
We think that we give in this way the first \prco of the result
(our \algo is \cot proved).

For a sub-$\R$-module $M$ of an  \Rmo $N$, where  $\R$
is an integral domain, the \textsl{saturation of $M$ in $N$} is the \Rmo
\[\Sat_{\R,N}(M)=\sotq{x\in N}{\exists a\in \R^*,\,ax\in M}.
\]
When there are several possible rings, we will precise ``the $\R$-saturation of $M$ in $N$''.
If  $N$
is a free module ($\R^n $, $n\in \NN$ or $\R^{(I)}$, $I$ infinite)
we get by scalar extension a \Kev $\K\otimes N\simeq\K^n $ or $\K^{(I)}$, where $\K$ is the quotient field of $\R$. We get also  $\Sat_{\R,N}(M)=\K.M\cap N$,
where  $\K.M$ is the sub-$\K$-\evc of $\K\otimes N$  generated by $M$.
The module $M$ is said to be  $\R$-saturated when it is equal to its $\R$-saturation.
In \gnl, the $\R$-saturation of a \tf $\RX$-module in $N=\RX^n $ is not itself a \tf $\RX$-module,
but this happens when $\R$ is a \ddv $\V$.

In Section \ref{satVmotf} we give an incremental \algo for calulating a basis of the $\V$-saturation of a
\tf sub-module of a free \Vmo (with a basis possibly infinite).
This \algo is not a scoop, but we give it in a form such that we are able to
use it, in Section \ref{satVXmotf}, for calculating a finite generating set of the $\VX$-module we get by  $\V$-saturating a \tf sub-$\VX$-module of
 $\VX^n$ ($n\in \NN^*$).

This proves that the $\V$-saturation of a \tf $\VX$-module inside
 $\VX^n$ is indeed a \tf $\VX$-module.

Finally, as an immediate corollary we obtain in
Section \ref{VXsysygies} the computation of a finite generating set for the syzygy module over $\VX$ of a finite family of vectors in $\VX^k$.

\medskip
In this paper, all rings are commutative and unitary.

\section{Saturation of a \tf \Vmo in a free \Vmo} \label{satVmotf}

\noindent  {\bf Terminology.}
In this paper we use the \cov terminology found in
 \citealt*{MRR}, which is well suited for Computer Algebra.

\smallskip For an arbitrary ring $\R$ we denote by $\R\eti$ the multiplicative group of units in $\R$.
The ring~$\R$ is \textsl{discrete} when we have an \algo deciding if  $x=0$ or $x\neq 0$ for an arbitrary \elt of~$\R$.
A ring~$\R$ is \textsl{local} when we have explicitly the implication
\[
\forall x,y\in\R,\,x+y\in\R\eti\; \Longrightarrow\;(x \in\R\eti\;\vu \;y\in\R\eti)
\]
It is equivalent to ask
$$
\forall x\in\R,(x \in\R\eti\;\vu \;1+x\in\R\eti)
$$
A nontrivial local ring $\R$ has a unique maximal ideal which is its \textsl{Jacobson radical}
\[\Rad(\R)=\sotq{x\in\R}{1+x\R\subseteq \R\eti}
\]
(for any ring, in \clama, this ideal is the intersection of \idemas).

The quotient ring $\k=\R/\Rad(\R)$ is a field, called \textsl{residual field} of $\R$.
The local ring $\R$ is said to be  \textsl{residually discrete} when we have explicitly the disjunction
\[\forall x\in\R,\,(x\in\R\eti\;\vu\;x\in\Rad(\R)) .
\] 
In this case, the residual field is a \cdi; we have an \algo deciding the disjunction
``$x=0$ or $x$ invertible'' for $x\in\k$.

\medskip Let $I$ be a set, finite or infinite, with a  ``discrete'' linear order, i.e. we have an \algo deciding the disjunction
\[i<j \quad \vu\quad  i=j \quad \vu\quad  i>j\]
for $i,j\in I$. In this case, for any ring $\R$, we have the free module $\R^{(I)}$ with the natural basis $(\e_i)_{i\in I}$. If $\R$ is discrete
 any vector $a=\sum_{i\in J} a_i\e_i$ in  $\R^{(I)}$
($J$ is a finite part of $I$) can be tested zero or nonzero, and if it is nonzero, we can find the minimal
$i$ for which $a_i\neq 0$.

In order to determine a saturated sub-\Rmo of a free module $\R^{(I)}$
we shall use  Lemmas~\ref{lemSat1} and~\ref{lemSat2}.

\begin{lemma} \label{lemSat1}
Let $\R$ be a domain and $N$ a torsion-free \Rmo. If $N=M\oplus P$
then $M$ is saturated in $N$.
\end{lemma}

 We now assume in this section that $\R$ is an integral residually discrete local ring with \hbox{$\k=\R/\Rad(\R)$}.
\begin{definota} \label{notaPiv}
{\rm  ~
\begin{enumerate}
\item A vector $C=\sum_i c_i \e_i$ of $\R^{(I)}$
 (seen as a column vector)
 is said to be \textsl{primitive} if it is \rdt nonzero, i.e., one \coe
 is a unit. In this case we denote by $\piv(C)$ the smallest index
 $i$ for which $c_i$ is \rdt nonzero. We call it \textsl{the pivot index of $C$}. We denote by $\cq(C)$
 (the \textsl{pivot coefficient of $C$}) the corresponding scalar $c_i$.
\item A finite family $(C^k)_{k\in K}$ of primitive vectors is said to be \textsl{in a $\k$-echelon form}
if the pivot indices $\piv(C^k)$ are distinct. A family in a $\k$-echelon form is also simply called \textsl{in an echelon form}.
In this case we denote
\[
  \piv(C)=\sotq{\piv(C^k)}{k\in K}.
\]
A matrix with \coes in $\R$ is said to be in an echelon form if its column vectors are primitive and form a  family in an echelon form.
\item When $K$ is linearly ordered, the family $(C^k)_{k\in K}$ is said to be \textsl{in a strict $\k$-echelon form} if it is in an echelon form and if moreover, for $j<k$ in $K$ the \coe of index $\piv(C^j)$ in $C^k$ is zero.
\end{enumerate}
  }
\end{definota}

\hum{Pour définir, and pour construire dans le cas d'un \ddv, une forme  $\k$-staggered
stricte, il n'y a pas besoin de supposer que l'anneau local est \dcd.
L'important est que chaque column soit affectée d'un indice pivot avec le \coe correspondant invertible, and que en suivant l'ordre de $K$, les \coes d'une column pour les indices pivots de columns précédentes soient nuls.
}

\begin{lemma} \label{lemSat2}\label{lemEchStrict}
If a finite family $C=(C^k)_{k\in K}$ of vectors of $\R^{(I)}$
is in a $\k$-echelon form, then it is~a basis of the \Rmo $M$ generated by itself  and $M$ has the following free \Rmo
as direct summand
\[
 P= \bigoplus\nolimits_{j\in J}\R\,\e_j\quad \hbox{avec } J=\sotq{j\in I}{j\notin \piv(C)}.
 \]
Moreover the set $\piv(C)$ does not depend on the module $M$. Indeed, an index $j$ belongs to $\piv(C)$ \ssi there exists a primitive vector $U$
in $M$ such that $\piv(U)=j$.
\end{lemma}
NB. The two free \Rmos in the lemma are always $\R$-saturated, a consequence of  Lemma~\ref{lemSat1}.
\begin{proof}
We give the proof when $I$ is finite, the \gnl case is similar.
We linearly order the family  $(\e_j)_{j\in J} \cup (C^{k})_{k\in K}$ using increasing pivot indices.
The  matrix we obtain is \rdt triangular with invertible \coes on the diagonal, so it is \rdt invertible, and it is invertible. This shows that  $M$ and $P$ are direct summands and admit the desired bases. The reader can finish the proof.
\end{proof}

\subsection*{The saturation \algo}

\begin{context} \label{context1}

\noindent Let $\V$ be a \ddv, i.e.\ a domain in which for all  $a,b$
either $a\mid b$ or~$b\mid a$, i.e.\ more precisely we have an algorithm deciding   (for $a,b$ given in $\V$) the disjunction
\[
\exists x\in\V,\,a=xb\quad \vu\quad \exists x\in\V,\,b=xa
\]
and gives the element  $x$.
We know that $\V$ is a local ring (supposing that $a+b$ is invertible, if~$a$ divides~$b$ then it divides  $a+b$ and thus it is invertible, if $b$ divides $a$ then it divides $a+b$ and is invertible).
We denote by $\K$ its quotient field and $\k$ its residual field.
As $\V$ is supposed to be integral in an explicit way, the field $\K$
is a discrete field.
Furthermore we suppose that $\V$ is  \textsl{residually discrete}, that is we have an algorithm deciding whether an element in  $\V$
is a unit. In particular Lemmas \ref{lemSat1} and \ref{lemSat2} hold with the ring   $\V$.
\end{context}

\medskip  \noindent \textsl{Remark.}
In a residually discrete valuation domain, we have a  test answering the question
 \gui{$a\mid b$?}. Indeed, for $a,b$ nonzero, if $a=bx$, then $a\mid b$ if and only if $x$ is a unit.

\medskip
We consider a sub-\Vmo $M=\V \,a^1+\cdots+\V \, a^m$ of $\V^{(I)}$.
The goal of this section is to give an algorithm for computing a basis for the  $\V$-saturation of $M$ in $\V^{(I)}$, a $\V$-module denoted, in short, by
$\Sat(M)$.
In fact, since only a finite number of indices are at stake, we can suppose that $I$ is finite and that $M$ is generated by the columns of a matrix $F$.
In order to visualise this we can write the  rows of the matrix accordingly to a nonincreasing order for the indices.

A rough way to  calculate $\Sat(M)$ could be to reduce  $F$
to the  Smith form by elementary operations. We thus see that after a suitable basis change, the module $M$
is generated by   vectors $v_if_i$ (where the $f_i$ are part of a basis and the $v_i$ are nonzero). In these conditions, the module $\Sat(M)$
is simply the module generated by these  $f_i$.
This tells us that $\Sat(M)$  is a free $\V$-module having as  direct summand another
free $\V$-module.

\smallskip
 In fact, we prefer to proceed in a softer way and obtain a basis for   $\Sat(M)$ as the columns vectors of a matrix $G$ in a
 $\k$-echelon form that we compute from   $F$ by means of very simple operations.

A first operation, of a reduction of a vector, that we denote by
$\RedPrim$ is to replace a vector~$a$, supposed nonzero, with
$a/u$, where $u$ is a gcd of its coefficients, for example $u$ is the coefficient with minimal index among those dividing all the others, in which case the  pivot coefficient of the reduced  vector is $1$.

The operations that we perform on the  matrix $F$ to put it in a strict
$\k$-echelon form $G$
such that $\Sat(\Im(F))= \Im(G)$  is the following. We proceed by dealing, one by one, with the columns of the initial matrix.
Note that at the beginning, the  empty matrix vide is in a strict $\k$-echelon form.

Suppose that we have treated some initial columns of the matrix and that we have obtained a matrix in a strict $\k$-echelon form with columns
$C^1,\dots,C^r$.

We want to treat a new  column, that we call $C=\sum_i c_i\e_i$.
We proceed as follows.
\begin{enumerate}
\item  Gaussian elimination:  we perform classical elementary operations on  columns
  $$
  C\leftarrow C- \frac{c_s}{c_{j,s}}\,C^j,
  $$
here $s=\piv(C^j)$
and $c_{j,s}=\cq(C^j)$.
This operation is made  successively with the columns $C^1,\dots,C^r$.
We then obtain a column $C'$.
\item If $C'=0$, we don't add it.
The matrix remains in a strict   $\k$-echelon form.
The \Kev generated by the columns $C^1,\dots,C^r,C$ has $(C^1,\dots,C^r)$ as basis.
\item If $C'\neq 0$, we replace $C'$ with its reduced primitive form
$C''=\RedPrim(C')$. We then add $C''$ as last  column $C^{r+1}$ of the matrix.
And the new matrix is in a strict  $\k$-echelon form.
The  \Kev generated by the columns $C^1,\dots,C^r,C$ has
$(C^1,\dots,C^r,C'')$ as basis.
\end{enumerate}

By construction $\Im(G)$ is contained in  $\Sat(\Im(F))$ and the \Vmo $\Im(G)$ is saturated since $G$ is in a strict $\k$-echelon form.
In fact, in both case 2. and case 3., we see by \recu that we have constructed  a basis for the $\V$-saturation generated by the first  columns (until the column $C$). At the end of this procedure we  have  $\Sat(\Im(G))=\Sat(\Im(F))$.
Thus our \algo achieves its goal.

\begin{theorem} \label{thAlgoSat}~

\noindent The saturation algorithm described above computes, from a matrix   $F$ with coefficients in $\V$, a matrix~$G$  in a strict $\k$-echelon form such that  $\Im(G)=\Sat(\Im(F))$.
\\
This algorithm is  \gui{incremental} in the following sense. If we treat a matrix
$[\,F_1\mid F_2\,]$, we obtain~a  matrix $[\,G_1\mid G_2\,]$ where $G_1$
is a matrix obtained by treating the  matrix $F_1$.
\end{theorem}

\medskip
The following lemma will be useful for the next section.

\begin{lemma} \label{lemPivG}
In the procedure ``Gaussian elimination'' described above, if $C$ is primitive and if the index $\piv(C)$ is distinct from the  $\piv(C^j)$
for $j=1,\dots,r$, then the obtained vector $C'$ is primitive with $\piv(C)=\piv(C')$.
\end{lemma}
%
\begin{proof}{}
We consider the affectation $C\leftarrow C- \frac{c_s}{c_{j,s}} C^j$ where $s=\piv(C^j)$
and $c_s$ is the coefficient of $C$ on the row $s$.
Set $\ell=\piv(C)$. The \coe $c_\ell$ on the row $\ell$ of $C$ is  replaced with  $c_\ell - \frac{c_s}{c_{j,s}}\cdot c_{j,\ell}$, where $c_{j,\ell}$ is the \coe  on the row $\ell $ of $C^j$.
If $\ell >s$, $c_s$
is \rdt null, if $\ell <s$ then $c_{j,\ell}$ is \rdt null, in both cases the \coe $c_\ell$ remains \rdt
unchanged.
\end{proof}
%

\section{The $\V$-saturation of a finitely generated  $\VX$-module} \label{satVXmotf}

The work we are doing  in this section is a little bit more delicate and seems, strangely, completely new.
It  achieves in Computer Algebra a simple theoretical result which is apparently new, and which  a fortiori lacks a constructive proof.

\begin{theorem} \label{thSat} We are in the context \ref{context1}.
If $M$ is a finitely generated sub-$\VX$-module of $\VX^n$ then the $\V$-saturation of
$M$ in  $\VX^n$ is a finitely generated  $\VX$-module.
\end{theorem}

\noindent {\it Remark.} Note that in classical mathematics, every \ddv satisfies the hypotheses of context~\ref{context1} by making use of the law of the excluded middle.
Hence, our constructive proof of \tho~\ref{thSat}  gives also a proof in
classical mathematics with the only hypothesis that $\V$
is a \ddv. The same remark applies for all the results of this  article.

\medskip The \dem of the \tho follows from the correctness
of the \algo computing  a finite generating set for the $\V$-saturation.

Seen as a $\VX$-module,  we have the natural basis of   $\VX^n$
denoted by $(\f_1,\dots,\f_n)$.
We are then interesed in  a natural basis of  $\VX^n$ as \Vmo, which is made up of $\e_{i,k}=X^k\,\f_i$ with the indices set  $I=\intervalle{1..n}\times \NN$.
We equip $I$ with the  lexicographic order for which
 $$
X^h\, \f_i < X^k\,\f_j \,\hbox{ if }\, i<j \;\hbox{ or } \;i=j \hbox{ and } h<k.
 $$

When the module $\VX^{n}$ is seen as a \Vmo with the natural basis given by the $X^k\,\f_j$’s, we are talking about the \gui{\coos} on this basis.
When it is seen as a  \VXmo  with the natural basis given by the~$\f_j$’s,
we are talking about the \gui{\coes} on this basis.

We start with a  list $S=[s^1,\dots,s^m]$ of vectors in
 $\VX^n$ which forms a  generating set of $M$.
We suppose \spdg that $m\geq 1$ and that the $s^k$ are nonzero.
We denote
$$
E=\V\,s^1+\cdots+\V\,s^m,\quad E_j=X^j\,E,\quad F'_k=\som_{j=0}^kE_j \quad\hbox{ and }\quad G'_k=\Sat_{\V,\VX^n}(F'_k).
$$

We can describe   $F'_k$ and $G'_k$ as the modules which are the images of the two matrices
$F_k$ and $G_k$. The matrix $F_k$ is given, it is treated with the saturation \algo
from the previous section, giving the matrix~$G_k$.

The question then arises as to  certify that from  some $k$,
there is no need to continue, because the \elts added in the basis of $G'_k$
 leave unchanged  the generated $\VX$-module (note that the \Vmo $G'_k=\Im(G_k)$ grows at each step as  $E\neq 0$).

\medskip We need to specify some  notation.
We call \gui{degree of $E$} and we denote it by $d$ the highest degree of one of the coordinates of one of the  $s^{k}$.
In the same way, $d+k$ will be the  degree of
$E_k$ or that of~$F_k$. the matrix $F_k$ can then be seen as a matrix
with $n(1+d+k)$ rows and $m(1+k)$ columns.

If $a$ is a $\V$-primitive  vector of $\VX^n$, and if $\piv(a)=(j,r)\in I$ we denote
$$
\idx(a) := j \hbox{ and } \PrimMon(a) :=r.
$$
The integer $\idx(a)$ is called the  \textsl{index of $a$}, the integer $\PrimMon(a)$
its \textsl{first residual exponent} and the couple  $\piv(a)$ is \textsl{the pivot index of  $a$}. Every couple $(j,r)\in I$ serves as an index for a vector  $X^{r}\,\f_{j}$ of the $\V$-natural basis of $\VX^n$.

We denote by $H_k$ the  matrix made by the columns that we add to  $G_{k-1}$
in order to obtain the   matrix $G_k$.

\begin{fact} \label{factHksuffit}
In order to  calculate the  matrix $G_{k+1}$ from the matrix $G_k$, instead of treating the  \gtrs of $E_{k+1}$ (i.e.\ the list $X^{k+1}S$), we can simply treat the
column vectors of $XH_k$.
\end{fact}
%
\begin{proof}{}
Let us consider the simplified procedure described above. \\
We denote by $\ov{G_k}$ the successive matrices obtained with this  simplified procedure.
\\
We easily check by  \recu on $k$ that the \Kev generated by  the columns
of $\ov{G_k}$ is the  sub-espace $\K F'_k$ of $\KX^{m}$.
Indeed, every column reduced to $0$ is in the  \Kev generated by he previous columns. And every  column which is not reduced to  $0$, generates, modulo the previous columns, once reduced, the same \Kev as the
column that initially formed it.
\\
The columns of $\ov{G_k}$ form a basis of a saturated
sub-\Vmo of $\VX^{m}$, which is then equal to   $\K F'_k\cap\VX^{m}$. This shows that $\ov{G_k}=G_k$
\end{proof}

In the sequel, we refer to the simplified procedure, but we denote $H_k$ and $G_k$ instead of $\ov{H_k}$ and $\ov{G_k}$.

To the matrices $H_k$ and $G_k$ we associate several integers:
\begin{itemize}
\item The integer $r_k$ is the number of columns of $G_k$, that is the rank of the free \Vmo  $\Im(G_k)$.
\item The integer $n_k$, \textsl{number of index pivots  present in $H_k$}, is the cardinality of the set of the  $i\in\intervalle{1..n}$ such that there exists a   column $C$ of $H_k$ with $\idx(C)=i$.
By Lemma~\ref{lemEchStrict}, all the index pivots present
in $H_{k-1}$ are present in $H_k$, thus, by \recu,
 $n_k$ is also the number of index pivots present in $G_k$. Therefore the sequence  $n_k$ is nondecreasing.
\item The integer $u_k$, \textsl{number of  \coos which are available for  $G_k$ in view of $n_k$}, is equal to  $n_k(1+d+k)$. If $n_{k+1}=n_k$
we have $u_{k+1}=u_k+n_k$.
\item The integer $\delta_k$, \textsl{defect of $H_k$}, is the number of  columns $C$
of $H_k$ such that there exists another column  $C'$ of $H_k$
with $\idx(C)=\idx(C')$
and $\PrimMon(C)<\PrimMon(C')$. Such a column $C$ will be called \textsl{supernumerary}.  We thus have $r_k\leq u_k$ and $r_k=r_{k-1}+n_k+\delta_k$.
\item The integer $\Delta_k=u_{k+1}-r_k$ is the \textsl{available place (or position) to occupy at step $k+1$
if $n_k=n_{k+1}$}.\\ If $n_k=n_{k+1}$, we have $\Delta_k=u_{k+1}-r_k=(u_k+n_k)-(r_{k-1}+n_k+\delta_k)=u_k-r_{k-1}-\delta_k=\Delta_{k-1}-\delta_k$.
\end{itemize}

\medskip
In order to visualize the defect $\delta_k$ of $H_k$ and  the way  it evolves
when $k$ increases we will use, after the proof, some
 figures illustrating what may happen. The reading of the proof shall be facilitated by the comments coming with the figures.

The essential point is the following.

From Lemma~\ref{lemPivG}, we are certain  that when we will treat the successive columns of $H_{k+1}=XH_k$ by means of $G_k$,
every pivot  $(j,r)$ of a column of $H_k$ will be shifted a place further,
i.e.\ in position $(j,r+1)$, as a pivot
of a column  of $H_{k+1}$, except  the case where the index
$(j,r+1)$ is already present in   $G_0$.
In this latter case, either the collision reduces to $0$ the column
of~$XH_k$ (which reduces the defect), or a new pivot  is occupied by the reduced   column 
(and made primitive). This new pivot  may have different effects.
Either it occurs on an  index which is already occupied, and does not reduce the defect, or occurs on an unoccupied  index, in which case,
the defect decreases by $1$ and the number $n_k$ increases in between   $n_k$ and $n_{k+1}$.
Thus, the first  claim of the following lemma is established.

\begin{lemma} \label{lemFinDeLalgoCertaine}
The sequence $\delta_k$ is nonincreasing. It certainly reaches  $0$ for $k$ sufficiently large.
\end{lemma}
%
\begin{proof}{}
We have already pointed out that if $n_{k+1}=n_k$ and $\delta_k>0$ then $\Delta_{k+1}<\Delta_k$. For a sufficiently large $k$ we thus get  $\delta_k=0$
or $n_{k+1}>n_k$. In the second case, we reproduce the previous  situation.
Since the sequence $n_k$ is bounded by $n$, this may happen only finitely many times.
\end{proof}
%

\begin{lemma} \label{lemFinDeLalgoCorrect}
If $\delta_k=0$ then the $\VX$-module generated by $G_k$ is the $\V$-saturation of the
$\VX$-module generated by the $s^j$’s given at the beginning. We can then stop the \algo.
\end{lemma}
%
\begin{proof}{}
Since the sequence $\delta_k$ is now null, it suffices to prove that
the columns of $H_{k+1}$ are in the $\VX$-module generated by the 
columns of $G_k$. But, by virtue of  Lemma~\ref{lemPivG}, the columns of~$H_{k+1}$
are in the  \Vmo $\Im(G_k)+X\Im(H_k)$.
\end{proof}

\medskip \noindent {\bf An example with figures.}

\smallskip
The figure 1 represents the index pivots of $G_0=H_0$.
The six white circles are the \elts of $\piv(H_0)$.

The black (full) circles or squares correspond to 
  \elts of the $\V$-basis where no  index pivot of $G_0$ is present.
The black squares are put for the indices of pivots which are still unoccupied: when a whole row is black we put squares to  emphasize further.
\\
In the present case we have then  $n=5$, $d=4$, $n_0=4$, $r_0=6$, $u_0=20$,
$\Delta_0=14$, $\delta_0=2$.

We represent by double white circles the \elts $\piv(H_0)$ of supernumerary columns,  corresponding to $\delta_0=2$.

\begin{figure}[ht]
\begin{center}
\includegraphics*[width=9cm]{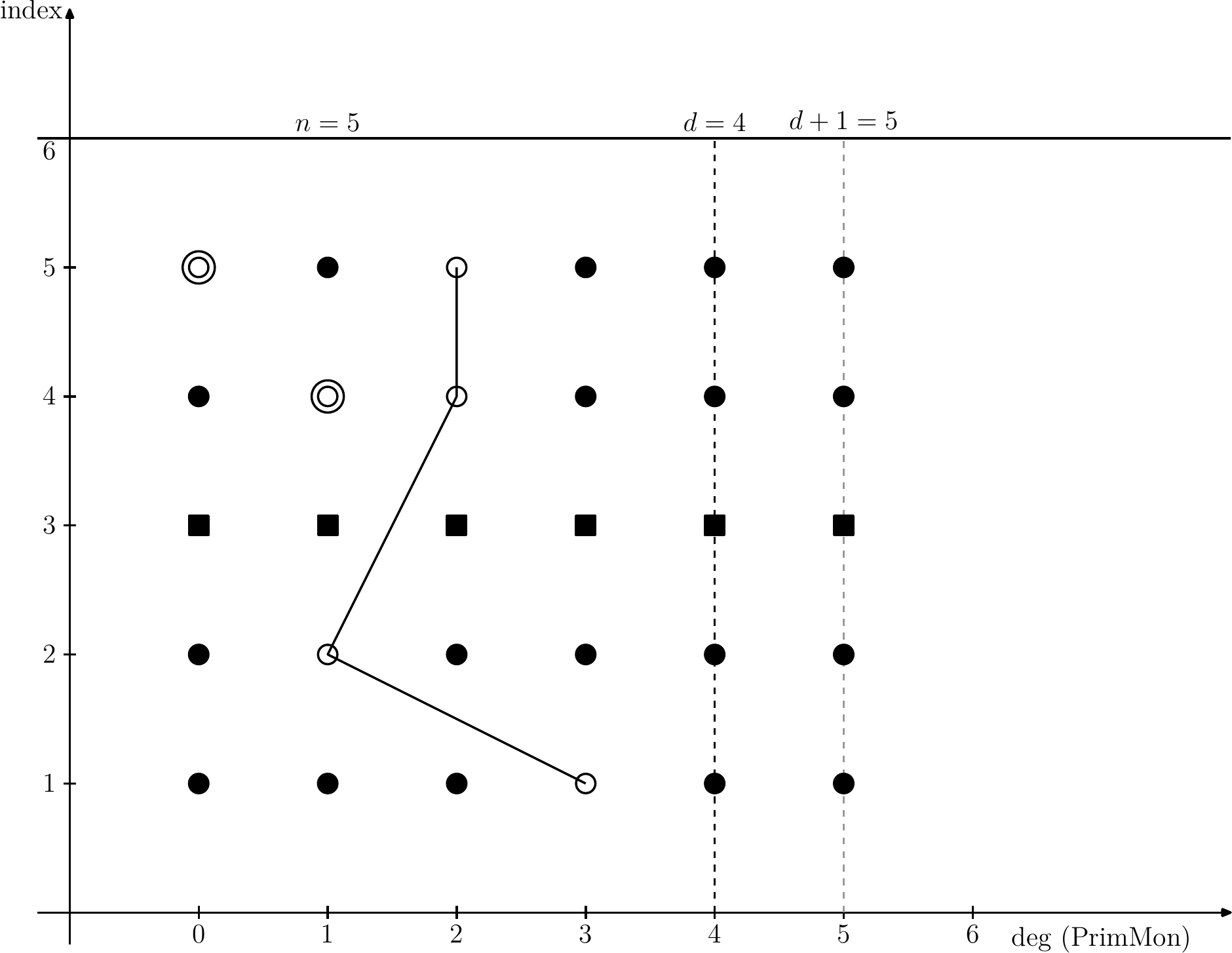}

\caption[figure 1]
{\label{fig1} }

\end{center}
\end{figure}

The broken black line joins the
$\piv(C)$’s of nonsupernumerary columns of $H_0$. \\
By Lemma~\ref{lemPivG}, we are certain that when we will treat the successive  columns of $H_1=XH_0$ by means of $G_0$,
every pivot  $(j,r)$ of a column of $H_0$ will be shifted a place further
i.e.\ in position $(j,r+1)$ as a  pivot of a 
column  of $H_1$, except the case where the index
$(j,r+1)$  is already present in  $G_0$.

In the case of figure 1, all the pivots of $H_0$ except the pivot $(4,1)$ are shifted a place further in~$H_1$. In particular the pivot $(5,1)$ will appear in $H_1$,
and we see that this will be for a supernumerary column (thus $\delta_1\geq 1$).
\\
When we will treat the 
column $XC$ such that $\piv(C)=(4,1)$, a \textsl{collision} occurs:
a Gaussian elimination will be performed in order to reduce to  $0$ the \coe in position $(4,2)$ of $XC$
and the  saturation procedure will produce, either a null vector
(in which case  $\delta_1=1$ and $H_1$ will have only $5$ columns, $r_1=11$), or a vector $C''$ such that $\piv(C'')$ occupies an unoccupied position in the  a priori available 
space (vectors of degree $\leq 5$)
with \ncrt  $\piv(C'')\notin \piv(G_0)$.
In this case we will have $r_1=12$.


Let us examine three  possibilities for this $\piv(C'')$, and give the corresponding three figures for~$H_1$.
The pivots of $H_0$ will be grey (empty) circles  and those of $H_1$ black (empty) circles.
The black (full) circles or squares correspond again to empty positions which could  a priori
be filled in the next steps.

In the case of  figure 2 we have  $n_1=4$, $\delta_1=2$, $u_1=24$,
$\Delta_1=12$.
When we will treat  $XH_1$ by means of $G_1$,
 pivots at positions $(1,5)$, $(2,4)$, $(4,4)$ and $(5,4)$
will be produced in  $H_2$. And two collisions, respectively at $(2,1)$ and $(5,2)$, will give  more difficult 
to predict results.
We can have $r_2=16$ with $\delta_2=0$, or  $r_2=17$ (with $\delta_2=1$ if $n_2=4$, or $\delta_2=0$ if $n_2=5$),
or also  $r_2=18$.

\begin{figure}[ht]
\begin{center}
\includegraphics*[width=9cm]{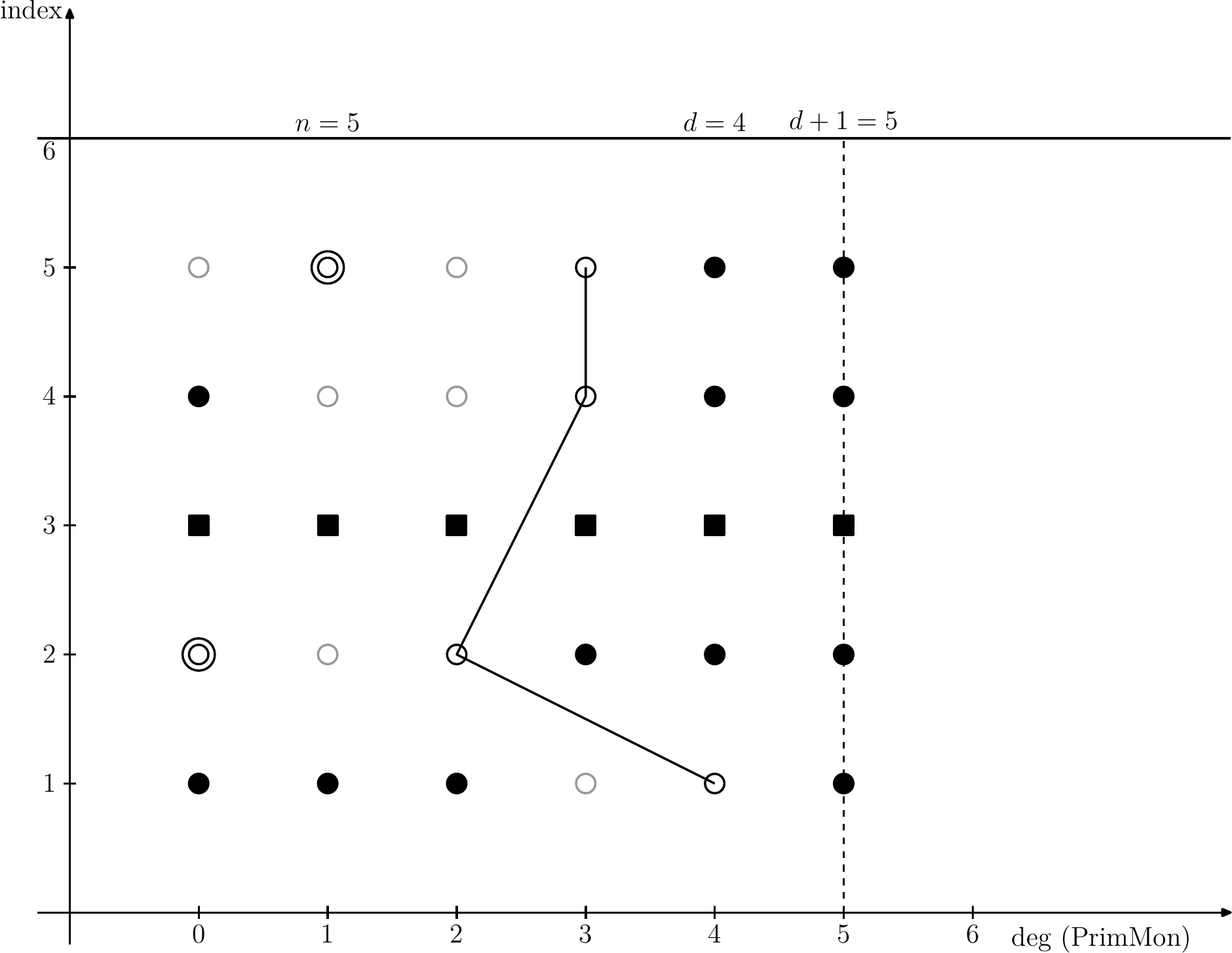}

\caption[figure 2]
{\label{fig2} If the collision at $(4,2)$ produces a pivot at position $(2,0)$ }

\end{center}
\end{figure}
	
\begin{figure}[ht]
\begin{center}
\includegraphics*[width=9cm]{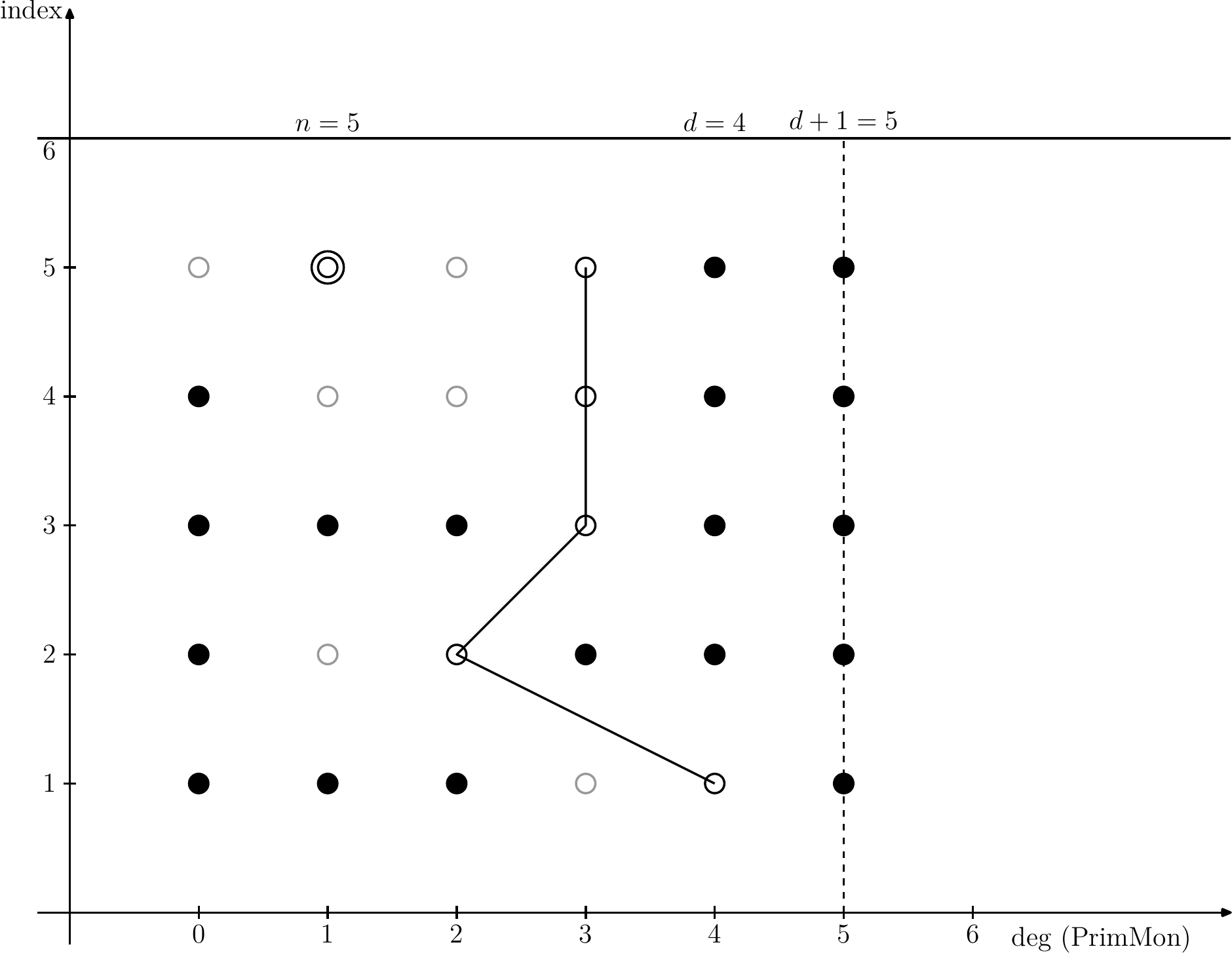}

\caption[figure 3]
{\label{fig3} If the collision at $(4,2)$ produces a pivot at $(3,3)$}

\end{center}
\end{figure}


In the case of figure 3  we have  $n_1=5$, $\delta_1=1$, $u_1=30$, $\Delta_1=18$.
When we will treat $XH_1$ by means of   $G_1$,
 pivots at positions $(1,5)$, $(2,3)$, $(3,4)$, $(4,4)$ and $(5,4)$
will be produced in $H_2$. And a collision, at $(5,2)$, will give a more difficult to predict result.
This could reduce the vector to $0$, in which case $\delta_2=0$, or produce a new  
vector, in which case $\delta_2=1$, because now all the  indices are occupied by the pivots.  The new vector will have a priori as pivot any of the black squares indicated in the figure, or also a
 pivot with degree $6$.

\begin{figure}[ht]
\begin{center}
\includegraphics*[width=9cm]{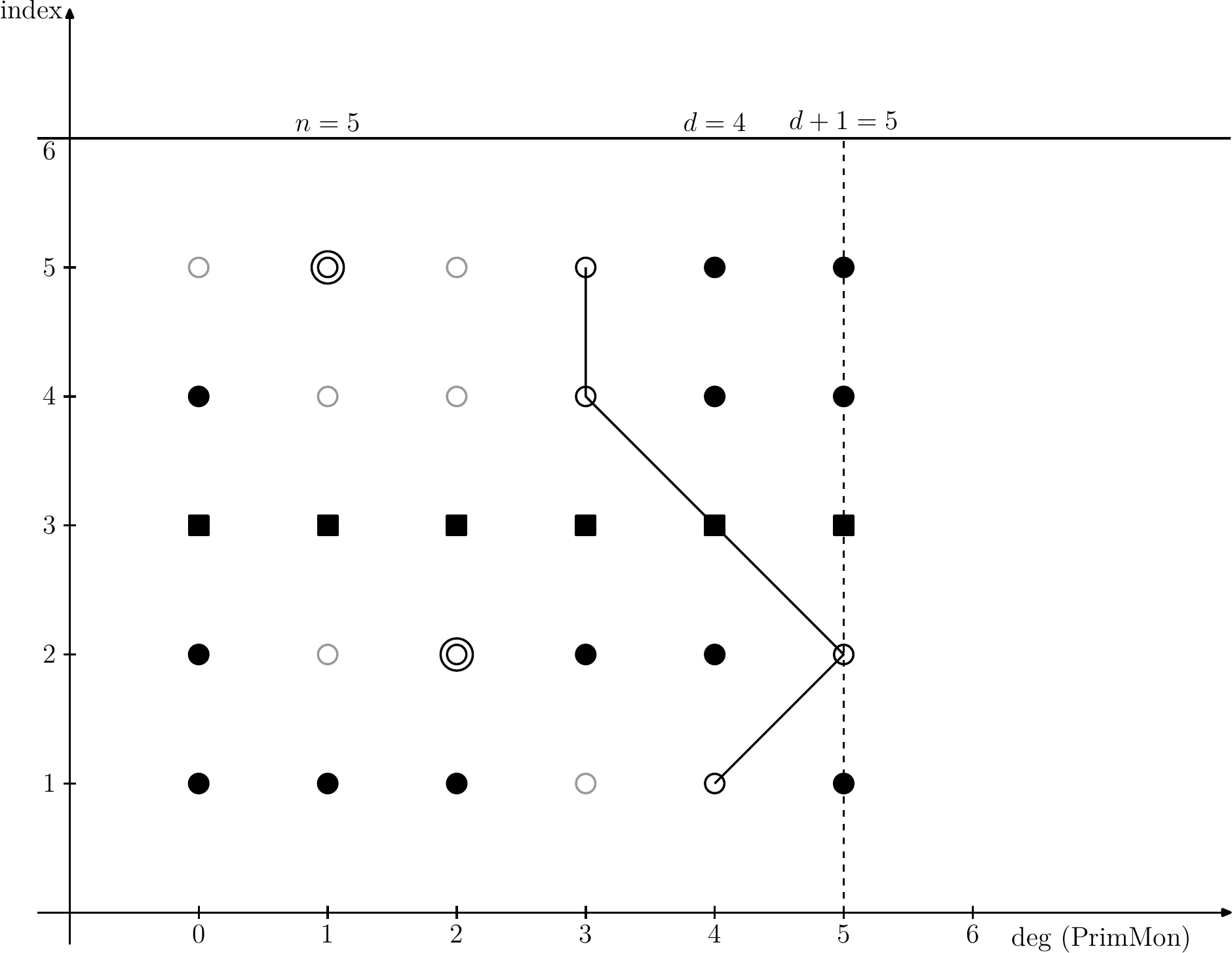}

\caption[figure 4]
{\label{fig4} If the collision at $(4,2)$ produces a pivot at $(2,5)$}

\end{center}
\end{figure}


In the case of figure 4 we have  $n_1=4$ and $\delta_1=2$.
When we will treat  $XH_1$ by means of $G_1$,
 pivots at positions $(1,5)$, $(2,4)$, $(4,4)$ and $(5,4)$
will be  produced in $H_2$. And a collision, at $(5,2)$, will give a more difficult to predict result.
The supernumerary column with pivot $(2,2)$ will not  a priori produce any collision,
except when the certain  collision, previously mentioned, is treated before and gives a reduced vector with pivot $(2,3)$.


\section{Syzygy module for a finitely generated  $\VX$-module} \label{VXsysygies}

\begin{theorem} \label{thsyzygies}
We keep the context \ref{context1}.
Let $u_1,\dots,u_n\in\VX^k$ and $s^1,\dots,s^m\in\KX^n$ be \gtrs of the syzygy
module of $(u_1,\dots,u_n)$ over $\KX$. We can suppose that
the $s^j$’s are in $\VX^n$. Then the syzygy  module of $(u_1,\dots,u_n)$ over $\VX$ is equal to the $\V$-saturation  in $\VX^n$ of the $\VX$-module generated by
$s^1,\dots,s^m$. As a consequence, by virtue of  \tho \ref{thSat}, this
 module is \tf.

\noindent In particular $\VX$ is a \coh ring.
\end{theorem}

%
\begin{proof}
A vector $f=(f_1,\dots,f_n)\in\VX^n$ such that $\sum_jf_ju_j=0$
can be written in the form
  $$
  f=a_1s^1+\cdots+a_ms^m
  $$
with some $a_j\in\KX$.
Multiplying this relation by a suitable $\alpha\neq 0$ in $\V$,
we obtain that~$\alpha\, a_i\in\VX$. This shows that $\alpha\,f$ is in the $\VX$-module generated by the $s^j$. And thus $f$ is in the 
$\V$-saturation of the $\VX$-module generated by the $s^j$.
The converse is straightforward.
\end{proof}

This leads to  the following \algo that computes a finite  generating set for the syzygy module of $u_1,\dots,u_n\in\VX^k$ over $\VX$:
\begin{enumerate}
\item Calculate vectors  $v^1,\dots,v^m\in\KX^n$ which form a finite generating set for the syzygy module of $(u_1,\dots,u_n)$ over $\KX$.
\item By multiplying each $v^j$ by a suitable $\alpha_j\in\K$, replace it with an $s^j\in\VX^n$ primitive.
\item Calculate by means of the \algo given in section \ref{satVXmotf} a finite generating set for 
the $\V$-saturation of the $\VX$-module  $\gen{s^1,\dots,s^m}$ dans $\VX^n$.
\end{enumerate}


\section{Annex:  Magma codes} \label{code}

We present in this annex some  magma codes for computing  the
 $\V$-saturation of a \tf $\VX$-sub-module of $\VX^{n}$
 (if $\V$ is a residually discrete \ddv) following the method described in the present article.

\smallskip Comments are  either on  a single line, preceded by {\tt //},
or on several lines, in between the signs {\tt /*} and {\tt */}.
\newpage
{\small 
\begin{verbatim}
%!TEX encoding =  UTF-8 Unicode

EchelonStrict := function(L, v0)
/*   L est une séquence strictement échelonnée de vecteurs (primitifs) de V[X]^n, 
     v0 un vecteur de V[X]^n non nul
     On note V.L le sous-V-module de V[X]^n engendré par L.
     Il est saturé dans V[X]^n  car la séquence L est strictement échelonnée.
     La fonction retourne un vecteur v et un booléen.
     Le booléen est faux si v est dans V.L + V.v0, vrai sinon
     Si v0 est dans V.L, alors v = 0
     Sinon, v est primitif, la séquence L' = (L, v) est strictement échelonnée
     et on a Sat(V.L + V.v0) = V.L'
*/

v := v0 ;
for w in L do 
   index, M, a := PivotAndCoefficient(w) ;
     // tuer le coefficient de v en position (index, M)
   v := v - a^(-1)*MonomialCoefficient(v[index], M)*w ;
    // si v est nul, c'est que v0 est dans L.V
   if v eq 0 then return v, false ; end if ;
end for ;

  // Ici, v est non nul. On le rend primitif
c := Contenu(v) ;   v := QuotientExact(v, c) ;
return v, not IsUnit(c) ;

end function ;

/*
S = s^1, ... , s^m est une séquence dans  V[X]^n
On considère dans V[X]^n le sous-V-module engendré par S , X*S , ... , X^k*S
et on note G'_k son saturé (comme V-module) (avec G'_k = {0} pour k < 0)
Dans la suite on calcule une V-base  G  de  G'_k  et  une liste  B  dans  V[X]^n
Cette liste  B,  extraite de  G,  suffira à engendrer, comme  V[X]-module,
le  V-saturé du V[X]-module engendré par  S.  
*/
/* 
Initialisations avec  G'_0
On calcule une V-base strictement échelonnée de G'_0
et le début de la liste  B  qui correspond à  G'_0
On initialise aussi  N,   qui est le nombre de vecteurs de  S  
qui restent en vie après cette première réduction 
*/

G := [Universe(S)| ] ; B:= G ;
// G est une liste vide d'objets du type de S
for v in S do 
   w , new := EchelonStrict(G,v) ; 
   if w ne 0 then Append(~G,w) ; end if ;
end for ;
B := G ; N := #G ;

/*  valeurs successives de G : V-base strictement échelonnée du V-module G'_k  
    valeurs successives de B : G allégée de vecteurs inutiles du point de vue 
                            du V[X]-moule engendré par  G'_k
    valeurs successives de N : dim_V (G'_k/G'_{k-1}) = nb de cols de H_k =
                            nombre de vecteurs de  X^k*S  qui restent en vie
                            après le calcul de G
*/

while true do

  // H <-- les N derniers vecteurs de G (= V-base de G'_{k-1}/G'_{k-2})
  H := G[#G-N+1 .. #G] ;

  // On calcule le défaut et on sort s'il est nul
  defaut := Defaut(H) ;

  if defaut eq 0 then break ; end if ;
  H := [X*v : v in H] ;
  N := 0 ;
  for v in H do
    w, new := EchelonStrict(G,v) ;
    if w ne 0 then Append(~G, w) ; N := N+1 ; end if ;
    if new then Append(~B, w) ; end if ;
  end for ;

end while ;
\end{verbatim}
}


\markboth{References}{References}

\addcontentsline{toc}{section}{References}
\small
\bibliographystyle{plainnat}

\normalsize

\endgroup
\stopcontents[english]


\clearpage
\newpage
\thispagestyle{empty}
~
\clearpage
\newpage

\setcounter{page}{1}\renewcommand\thepage{F\arabic{page}}\renewcommand\theHsection{F\arabic{section}}

\clearpage

\setcounter{section}{0}

\selectlanguage{french}
\def\frenchproofname{\textsl{Démonstration}}
\FrenchFootnotes

\startcontents[french]

\title{Un algorithme pour le calcul des syzygies sur  $\mathbf{V}[X]$\\
dans le cas où $\mathbf{V}$ est un domaine de valuation}
\author{
Henri Lombardi, 
Claude Quitté, Ihsen Yengui}

\date{2014}

\def\thefootnote{\arabic{footnote}}

\newcommand\tsf[1]{\textbf{\textsf{#1}}}
\newcommand\Vrai{\mathsf{Vrai}}
\newcommand\Faux{\mathsf{Faux}}
\newcommand\ET{\mathsf{ et }}
\newcommand\OU{\mathsf{ ou }}
\newcommand\sialors[1]{\tsf{si } $#1$ \tsf{ alors }}
\newcommand\tantque[1]{\tsf{tant que } $#1$ \tsf{ faire }}
\newcommand\finpour{\tsf{fin pour}}
\newcommand\sinon{\tsf{sinon }}
\newcommand\finsi{\tsf{fin si }}
\newcommand\fintantque{\tsf{fin tant que }}
\newcommand\Debut{\\[1mm] \tsf{D\'ebut }}
\newcommand\Fin{\tsf{\\ Fin.}}
\newcommand\fin{\tsf{\\ Fin.}}
\newcommand\Entree{\\[1mm] \tsf{Entr\'ee : }}
\newcommand\Sortie{\\ \tsf{Sortie : }}
\newcommand\Varloc{\\ \tsf{Variables locales : }}



\newcounter{MF}
\newcommand\stMF{\stepcounter{MF}}

\newcommand{\lec}{\stMF\ifodd\value{MF}lecteur \else 
lectrice \fi}
\newcommand{\lecz}{\stMF\ifodd\value{MF}lecteur\else lectrice\fi}

\newcommand{\lecs}{\stMF\ifodd\value{MF}lecteurs \else 
lectrices \fi}
\newcommand{\lecsz}{\stMF\ifodd\value{MF}lecteurs\else 
lectrices\fi}

\newcommand{\alec}{\stMF\ifodd\value{MF}au lecteur \else%
à la lectrice \fi}
\newcommand{\alecz}{\stMF\ifodd\value{MF}au lecteur\else%
à la lectrice\fi}

\newcommand{\dlec}{\stMF\ifodd\value{MF}du lecteur \else%
de la lectrice \fi}
\newcommand{\dlecz}{\stMF\ifodd\value{MF}du lecteur\else%
de la lectrice\fi}

\newcommand{\llec}{\stMF\ifodd\value{MF}le lecteur \else la lectrice \fi}
\newcommand{\llecz}{\stMF\ifodd\value{MF}le lecteur\else la lectrice\fi}

\newcommand{\Llec}{\stMF\ifodd\value{MF}Le lecteur \else La lectrice \fi}

\newcommand{\lui}{\ifodd\value{MF}lui \else
elle \fi}
\newcommand{\luiz}{\ifodd\value{MF}lui\else
elle\fi}

\newcommand{\celui}{\ifodd\value{MF}celui \else
celle \fi}

\newcommand{\ceux}{\ifodd\value{MF}ceux \else
celles \fi}

\newcommand{\er}{\ifodd\value{MF}er \else
ère \fi}

\newcommand{\eux}{\ifodd\value{MF}eux \else
elles \fi}

\newcommand{\eUx}{\ifodd\value{MF}eux \else
euse \fi}

\newcommand{\leux}{\ifodd\value{MF}leux \else
leuse \fi}

\newcommand{\il}{\ifodd\value{MF}il \else
elle \fi}

\newcommand{\ien}{\ifodd\value{MF}ien \else
ienne \fi}

\newcommand{\ez}{\ifodd\value{MF}\else e\fi}

\newcommand{\n}{\ifodd\value{MF}n \else nne \fi}
\newcommand{\nz}{\ifodd\value{MF}n\else nne\fi}

\makeatletter
\newcommand{\la}{\@ifstar{\ifodd\value{MF}le\else
la\fi}{\stMF\ifodd\value{MF}le \else la \fi}}
\makeatother

\newcommand \rem{\rdb
\noi{\sl Remarque. }}

\newcommand \REM[1]{\rdb
\noi{\sl Remarque#1. }}

\newcommand \rems{\rdb
\noi{\sl Remarques. }}

\newcommand \exl{\rdb
\noi{\bf Exemple. }}

\newcommand \EXL[1]{\rdb
\noi{\bf Exemple: #1. }}

\newcommand \exls{\rdb
\noi{\bf Exemples. }}

\newcommand \thref[1] {théorème~\ref{#1}}
\newcommand \paref[1] {page~\pageref{#1}}
\newcommand \pstfref[1] {Posi\-tiv\-stel\-lensatz formel~\ref{#1}}
\newcommand \pstref[1] {Posi\-tiv\-stel\-lensatz~\ref{#1}}

\newcommand\oge{\leavevmode\raise.3ex\hbox{$\scriptscriptstyle\langle\!\langle\,$}}
\newcommand\feg{\leavevmode\raise.3ex\hbox{$\scriptscriptstyle\,\rangle\!\rangle$}}
\newcommand\gui[1]{\oge{#1}\feg}

\newcommand \facile{\begin{proof}
La démonstration est laissée \alecz.
\end{proof}
}

\newcommand\comm{\rdb
\noi{\it Commentaire. }}

\newcommand\COM[1]{\rdb
\noi{\it Commentaire #1. }}

\newcommand\comms{\rdb
\noi{\it Commentaires. }}

\newcommand\Pb{\rdb
\noi{\bf Problème. }}

\newcommand\eoq{\hbox{}\nobreak
\vrule width 1.4mm height 1.4mm depth 0mm}

\newcommand \Cad {C'est-à-dire\xspace}
\newcommand \recu {récur\-rence\xspace}
\newcommand \hdr {hypo\-thèse de \recu}
\newcommand \cad {c'est-à-dire\xspace}
\newcommand \cade {c'est-à-dire en\-co\-re\xspace}
\newcommand \ssi {si, et seu\-lement si, }
\newcommand \ssiz {si, et seu\-lement si,~}
\newcommand \cnes {con\-di\-tion néces\-sai\-re et suf\-fi\-san\-te\xspace}
\newcommand \spdg {sans per\-te de géné\-ra\-lité\xspace}
\newcommand \Spdg {Sans per\-te de géné\-ra\-lité\xspace}

\newcommand \Propeq {Les pro\-pri\-é\-tés sui\-van\-tes sont 
équi\-va\-len\-tes.}
\newcommand \propeq {les pro\-pri\-é\-tés sui\-van\-tes sont 
équi\-va\-len\-tes.}

\newcommand \Kev {$\gK$-\evc}
\newcommand \Kevs {$\gK$-\evcs}

\newcommand \Lev {$\gL$-\evc}
\newcommand \Levs {$\gL$-\evcs}

\newcommand \Qev {$\QQ$-\evc}
\newcommand \Qevs {$\QQ$-\evcs}

\newcommand \kev {$\gk$-\evc}
\newcommand \kevs {$\gk$-\evcs}

\newcommand \lev {$\gl$-\evc}
\newcommand \levs {$\gl$-\evcs}

\newcommand \Alg {$\gA$-\alg}
\newcommand \Algs {$\gA$-\algs}

\newcommand \Blg {$\gB$-\alg}
\newcommand \Blgs {$\gB$-\algs}

\newcommand \Clg {$\gC$-\alg}
\newcommand \Clgs {$\gC$-\algs}

\newcommand \klg {$\gk$-\alg}
\newcommand \klgs {$\gk$-\algs}

\newcommand \llg {$\gl$-\alg}
\newcommand \llgs {$\gl$-\algs}

\newcommand \Klg {$\gK$-\alg}
\newcommand \Klgs {$\gK$-\algs}

\newcommand \Llg {$\gL$-\alg}
\newcommand \Llgs {$\gL$-\algs}

\newcommand \QQlg {$\QQ$-\alg}
\newcommand \QQlgs {$\QQ$-\algs}

\newcommand \Rlg {$\gR$-\alg}
\newcommand \Rlgs {$\gR$-\algs}

\newcommand \RRlg {$\RR$-\alg}
\newcommand \RRlgs {$\RR$-\algs}

\newcommand \ZZlg {$\ZZ$-\alg}
\newcommand \ZZlgs {$\ZZ$-\algs}

\newcommand \Amo {$\gA$-module\xspace}
\newcommand \Amos {$\gA$-modules\xspace}

\newcommand \Bmo {$\gB$-module\xspace}
\newcommand \Bmos {$\gB$-modules\xspace}

\newcommand \Cmo {$\gC$-module\xspace}
\newcommand \Cmos {$\gC$-modules\xspace}

\newcommand \kmo {$\gk$-module\xspace}
\newcommand \kmos {$\gk$-modules\xspace}

\newcommand \Kmo {$\gK$-module\xspace}
\newcommand \Kmos {$\gK$-modules\xspace}

\newcommand \Lmo {$\gL$-module\xspace}
\newcommand \Lmos {$\gL$-modules\xspace}

\newcommand \Rmo {$\gR$-module\xspace}
\newcommand \Rmos {$\gR$-modules\xspace}

\newcommand \Vmo  {$\gV$-module\xspace}
\newcommand \Vmos {$\gV$-modules\xspace}

\newcommand \VXmo  {$\VX$-module\xspace}

\newcommand \Ali {appli\-ca\-tion $\gA$-\lin}
\newcommand \Alis {appli\-ca\-tions $\gA$-\lins}

\newcommand \Kli {appli\-ca\-tion $\gK$-\lin}
\newcommand \Klis {appli\-ca\-tions $\gK$-\lins}

\newcommand \Bli {appli\-ca\-tion $\gB$-\lin}
\newcommand \Blis {appli\-ca\-tions $\gB$-\lins}

\newcommand \Cli {appli\-ca\-tion $\gC$-\lin}
\newcommand \Clis {appli\-ca\-tions $\gC$-\lins}

\newcommand \ac{algé\-bri\-quement clos\xspace}  

\newcommand \acl {an\-neau \icl}
\newcommand \acls {an\-neaux \icl}

\newcommand \adp {an\-neau de Pr\"u\-fer\xspace}
\newcommand \adps {an\-neaux de Pr\"u\-fer\xspace}

\newcommand \adpc {\adp \coh}
\newcommand \adpcs {\adps \cohs}

\newcommand \adu {\alg de décom\-po\-sition univer\-selle\xspace}
\newcommand \adus {\algs de décom\-po\-sition univer\-selle\xspace}

\newcommand \adv {anneau de valuation\xspace}
\newcommand \advs {anneaux de valuation\xspace}

\newcommand \advl {anneau \dvla} 
\newcommand \advls {anneaux \dvlas} 

\newcommand \Afr {Anneau \frl}
\newcommand \Afrs {Anneaux \frls}
\newcommand \afr {anneau \frl}
\newcommand \aFr {\hyperref[theorieAfr]{anneau \frl}\xspace}
\newcommand \afrs {anneaux \frls}

\newcommand \afrr {\afr réduit\xspace}
\newcommand \afrrs {\afrs réduits\xspace}
\newcommand \Afrrs {\Afrs réduits\xspace}

\newcommand \afrvr {\afr avec \ravs}
\newcommand \aFrvr {\hyperref[theorieAfrrv]{\afrvr}\xspace}
\newcommand \afrvrs {\afrs avec \ravs}

\newcommand \aftr {anneau réticulé \ftm réel\xspace}
\newcommand \aftrs {anneaux réticulés \ftm réels\xspace}

\newcommand \aG {\alg galoisienne\xspace}
\newcommand \aGs {\algs galoisiennes\xspace}

\newcommand \agB {\alg de Boole\xspace}
\newcommand \agBs {\algs de Boole\xspace}

\newcommand \agH {\alg de Heyting\xspace}
\newcommand \agHs {\algs de Heyting\xspace}

\newcommand \agq{algé\-bri\-que\xspace}
\newcommand \agqs{algé\-bri\-ques\xspace}

\newcommand \agqt{algé\-bri\-que\-ment\xspace}

\newcommand \aKr {anneau de Krull\xspace}
\newcommand \aKrs {anneaux de Krull\xspace}

\newcommand \alg {algè\-bre\xspace}
\newcommand \algs {algè\-bres\xspace}

\newcommand \algo{algo\-rithme\xspace}
\newcommand \algos{algo\-rithmes\xspace}

\newcommand \algq{al\-go\-rith\-mi\-que\xspace}
\newcommand \algqs{al\-go\-rith\-mi\-ques\xspace}

\newcommand \ali {appli\-ca\-tion \lin}
\newcommand \alis {appli\-ca\-tions \lins}

\newcommand \alo {an\-neau lo\-cal\xspace}
\newcommand \alos {an\-neaux lo\-caux\xspace}

\newcommand \algb {an\-neau \lgb}
\newcommand \algbs {an\-neaux \lgbs}

\newcommand \alrd {\alo \dcd}
\newcommand \alrds {\alos \dcds}

\newcommand \anar {anneau \ari}
\newcommand \anars {anneaux \aris}

\newcommand \anor {an\-neau nor\-mal\xspace}
\newcommand \anors {an\-neaux nor\-maux\xspace}

\newcommand \apf {\alg \pf}
\newcommand \apfs {\algs \pf}

\newcommand \apG {\alg pré\-galoisienne\xspace}
\newcommand \apGs {\algs pré\-galoisiennes\xspace}

\newcommand \arc {anneau réel clos\xspace}
\newcommand \aRc {\hyperref[theorieArc]{\arc}\xspace}
\newcommand \arcs {anneaux réels clos\xspace}

\newcommand \ari{arith\-mé\-tique\xspace}  
\newcommand \aris{arith\-mé\-tiques\xspace}  

\newcommand \Asr {Anneau \str}
\newcommand \Asrs {Anneaux \strs}
\newcommand \asr {anneau \str}
\newcommand \asrs {anneaux \strs}

\newcommand \asrvr {\asr avec \ravs}
\newcommand \asrvrs {\asrs avec \ravs}

\newcommand \atf {\alg \tf}
\newcommand \atfs {\algs \tf}

\newcommand \auto {auto\-mor\-phisme\xspace}
\newcommand \autos {auto\-mor\-phismes\xspace}


\newcommand \bdg {base de Gr\"obner\xspace}
\newcommand \bdgs {bases de Gr\"obner\xspace}

\newcommand \bdp {base de \dcn partielle\xspace}
\newcommand \bdps {bases de \dcn partielle\xspace}

\newcommand \bdf {base de \fap\xspace}

\newcommand \Bif {Borne infé\-rieure\xspace} %
\newcommand \bif {borne infé\-rieure\xspace} %
\newcommand \bifs {bornes infé\-rieures\xspace} %

\newcommand \bsp {borne supé\-rieure\xspace} %
\newcommand \bsps {borne supé\-rieures\xspace} %


\newcommand \cac{corps \ac}  

\newcommand \calf{calcul formel\xspace}  

\newcommand \cara{carac\-té\-ris\-tique\xspace}  
\newcommand \caras{carac\-té\-ris\-tiques\xspace}  

\newcommand \carn{carac\-té\-ri\-sation\xspace}  
\newcommand \carns{carac\-té\-ri\-sations\xspace}  

\newcommand \carar{carac\-té\-riser\xspace}

\newcommand \carf{de carac\-tère fini\xspace}  

\newcommand \cdi{corps discret\xspace}
\newcommand \cdis{corps discrets\xspace}
  
\newcommand \cdv{changement de variables\xspace}  
\newcommand \cdvs{changements de variables\xspace}  

\newcommand \cli {clô\-ture inté\-grale\xspace}

\newcommand \codi {corps ordonné discret\xspace}
\newcommand \codis {corps ordonnés discrets\xspace}

\newcommand \coe {coef\-fi\-cient\xspace}
\newcommand \coes {coef\-fi\-cients\xspace}

\newcommand \coh {co\-hé\-rent\xspace}
\newcommand \cohs {co\-hé\-rents\xspace}

\newcommand \cohc {co\-hé\-rence\xspace}

\newcommand \coli {combi\-nai\-son \lin}
\newcommand \colis {combi\-nai\-sons \lins}

\newcommand \com {co\-ma\-xi\-maux\xspace}
\newcommand \come {co\-ma\-xi\-males\xspace}

\newcommand \coo {coor\-donnée\xspace}
\newcommand \coos {coor\-données\xspace}

\newcommand \cop {complé\-men\-taire\xspace}
\newcommand \cops {complé\-men\-taires\xspace}

\newcommand \cosv {conser\-vative\xspace}
\newcommand \cosvs {conser\-vatives\xspace}

\newcommand \cOsv {\hyperref[defithconserv]{conser\-vative}\xspace}
\newcommand \cOsvs {\hyperref[defithconserv]{conser\-vatives}\xspace}

\newcommand \covr {corps ordonné avec \ravs}
\newcommand \covrs {corps ordonnés avec \ravs}

\newcommand \cpb {compa\-tible\xspace} 
\newcommand \cpbs {compa\-tibles\xspace} 

\newcommand \cpbt {compa\-tibi\-lité\xspace} 
\newcommand \cpbtz {compa\-tibi\-lité} 

\newcommand \crc {corps réel clos\xspace}
\newcommand \crcs {corps réels clos\xspace}

\newcommand \crcd {corps réel clos discret\xspace}
\newcommand \crcds {corps réels clos discrets\xspace}


\newcommand \dcd {rési\-duel\-lement dis\-cret\xspace}
\newcommand \dcds {rési\-duel\-lement dis\-crets\xspace}

\newcommand \dcn {décom\-po\-sition\xspace}
\newcommand \dcns {décom\-po\-sitions\xspace}

\newcommand \dcnb {\dcn bornée\xspace}

\newcommand \dcnc {\dcn complète\xspace}

\newcommand \dcnp {\dcn partielle\xspace}

\newcommand \dcp {décom\-posa\-ble\xspace}
\newcommand \dcps {décom\-posa\-bles\xspace}

\newcommand \ddk {dimension de~Krull\xspace}
\newcommand \ddi {de dimension infé\-rieure ou égale à~}

\newcommand \ddp {domaine de Pr\"u\-fer\xspace}
\newcommand \ddps {domaines de Pr\"u\-fer\xspace}

\newcommand \ddv {domaine de valuation\xspace}
\newcommand \ddvs {domaines de valuation\xspace}

\newcommand \Demo{Démonstration\xspace}     

\newcommand \demo{démon\-stra\-tion\xspace}     
\newcommand \demos{démon\-stra\-tions\xspace}     

\newcommand \dem{démons\-tra\-tion\xspace}
\newcommand \dems{démons\-tra\-tions\xspace}

\newcommand \deno{déno\-mi\-nateur\xspace}     
\newcommand \denos{déno\-mi\-nateurs\xspace}   

\newcommand \deter {déter\-mi\-nant\xspace}  
\newcommand \deters {déter\-mi\-nants\xspace}  
  
\newcommand \Dfn{Défi\-nition\xspace}  
\newcommand \Dfns{Défi\-nitions\xspace}  
\newcommand \dfn{défi\-nition\xspace}  
\newcommand \dfns{défi\-nitions\xspace}  

\newcommand \dftr {droite réticulée \ftm réelle\xspace}
\newcommand \dftrs {droites réticulées \ftm réelles\xspace}
  
\newcommand \dil{diffé\-rentiel\xspace}  
\newcommand \dils{diffé\-rentiels\xspace}  
\newcommand \dile{diffé\-ren\-tielle\xspace}  
\newcommand \diles{diffé\-ren\-tielles\xspace}  

\newcommand \dip{diviseur principal\xspace}
\newcommand \dips{diviseurs principaux\xspace}

\newcommand \discri{discri\-minant\xspace}  
\newcommand \discris{discri\-minants\xspace}  

\newcommand \divle {dimension divisorielle\xspace}

\newcommand \dit{distri\-bu\-ti\-vité\xspace}

\newcommand \dlg{d'élar\-gis\-sement\xspace}  

\newcommand \dok {domaine de Dedekind\xspace}
\newcommand \doks {domaines de Dedekind\xspace}

\newcommand \dvla {à diviseurs\xspace}
\newcommand \dvlas {à diviseurs\xspace}

\newcommand \dvld {\dvlt décom\-posé\xspace} %
\newcommand \dvlds {\dvlt décom\-posés\xspace} %

\newcommand \dvlg {diviso\-riel\xspace} 
\newcommand \dvlgs {diviso\-riels\xspace} 

\newcommand \dvli {\dvlt inver\-sible\xspace} 
\newcommand \dvlis {\dvlt inver\-sibles\xspace} 

\newcommand \dvlt {diviso\-riel\-lement\xspace} %

\newcommand \dvz {di\-viseur de zéro\xspace}
\newcommand \dvzs {di\-viseurs de zéro\xspace}

\newcommand \dve {divi\-si\-bi\-lité\xspace}

\newcommand \dvee {à \dve explicite\xspace}

\newcommand \dvr {diviseur\xspace}
\newcommand \dvrs {diviseurs\xspace}


\newcommand \Eds {Exten\-sion des sca\-laires\xspace}
\newcommand \edss {exten\-sions des sca\-laires\xspace}
\newcommand \eds {exten\-sion des sca\-laires\xspace}

\newcommand \eco {\elts \com}

\newcommand \egmt {éga\-lement\xspace}

\newcommand \egt {éga\-li\-té\xspace}
\newcommand \egts {éga\-li\-tés\xspace}

\newcommand \eli{élimi\-nation\xspace}  

\newcommand \elr{élé\-men\-taire\xspace}  
\newcommand \elrs{élé\-men\-taires\xspace}  

\newcommand \elrt{élé\-men\-tai\-rement\xspace}  

\newcommand \elt{élé\-ment\xspace}  
\newcommand \elts{élé\-ments\xspace}  

\def \endo {en\-do\-mor\-phisme\xspace}
\def \endos {en\-do\-mor\-phismes\xspace}

\newcommand \entrel {rela\-tion impli\-ca\-tive\xspace}
\newcommand \entrels {rela\-tions impli\-ca\-tives\xspace}

\newcommand\evc{es\-pa\-ce vec\-to\-riel\xspace} 
\newcommand\evcs{es\-pa\-ces vec\-to\-riels\xspace} 

\newcommand \eqv {équi\-valent\xspace}  
\newcommand \eqve {équi\-va\-lente\xspace}  
\newcommand \eqvs {équi\-valents\xspace}  
\newcommand \eqves {équi\-val\-entes\xspace}  

\newcommand \eqvc {équi\-va\-lence\xspace}  
\newcommand \eqvcs {équi\-va\-lences\xspace}  

\newcommand \esid {essen\-tiel\-lement iden\-tique\xspace}  
\newcommand \esids {essen\-tiel\-lement iden\-tiques\xspace}  

\newcommand \Esid {\hyperref[defitdyesidentiques]{\esid}\xspace}  
\newcommand \Esids {\hyperref[defitdyesidentiques]{\esids}\xspace}  

\newcommand \eseq {essen\-tiel\-lement \eqve}  
\newcommand \eseqs {essen\-tiel\-lement \eqves}  

\newcommand \Eseq {\hyperref[defitheseq]{\eseq}\xspace}  
\newcommand \Eseqs {\hyperref[defitheseq]{\eseqs}\xspace}

\newcommand \fab {\fcn bornée\xspace}
\newcommand \fabs {\fcns bornées\xspace}

\newcommand \fat {\fcn totale\xspace}
\newcommand \fats {\fcn totales\xspace}

\newcommand \fap {\fcn partielle\xspace}
\newcommand \faps {\fcns partielles\xspace}

\newcommand \fip {filtre pre\-mier\xspace}
\newcommand \fips {filtres pre\-miers\xspace}

\newcommand \fipma {\fip maxi\-mal\xspace}
\newcommand \fipmas {\fips maxi\-maux\xspace}

\newcommand \fcn {factorisation\xspace}
\newcommand \fcns {factorisations\xspace}

\newcommand \fdi {for\-te\-ment dis\-cret\xspace}
\newcommand \fdis {for\-te\-ment dis\-crets\xspace}

\newcommand \fsa {fermé \sagq}
\newcommand \fsas {fermés \sagqs}

\newcommand \fsagc {fonction \sagc}
\newcommand \fsagcs {fonctions \sagcs}

\newcommand \fmt {formel\-lement\xspace}

\newcommand \frl {for\-tement réticulé\xspace}
\newcommand \frle {for\-tement réticulée\xspace}
\newcommand \frls {for\-tement réticulés\xspace}

\newcommand \ftm {fortement\xspace}

\newcommand\gmt{géométrie\xspace}  
\newcommand\gmts{géométries\xspace}  

\newcommand\gaq{\gmt \agq}  

\newcommand\gmq{géomé\-trique\xspace}  
\newcommand\gmqs{géomé\-triques\xspace}  

\newcommand\gmqt{géomé\-tri\-quement\xspace}  

\newcommand\gne{géné\-ra\-lisé\xspace}  
\newcommand\gnee{géné\-ra\-lisée\xspace}  
\newcommand\gnes{géné\-ra\-lisés\xspace}  
\newcommand\gnees{géné\-ra\-lisées\xspace}  

\newcommand\gnl{géné\-ral\xspace}  
\newcommand\gnle{géné\-rale\xspace}  
\newcommand\gnls{géné\-raux\xspace}  
\newcommand\gnles{géné\-rales\xspace}  

\newcommand\gnlt{géné\-ra\-lement\xspace}  

\newcommand\gnn{géné\-ra\-li\-sa\-tion\xspace}  
\newcommand\gnns{géné\-ra\-li\-sa\-tions\xspace}  

\newcommand\gnq {géné\-rique\xspace}  
\newcommand\gnqs {géné\-riques\xspace}  

\newcommand\gnr{géné\-ra\-liser\xspace}  

\newcommand \gns{géné\-ra\-lise\xspace}

\newcommand \gnt{géné\-ra\-lité\xspace}
\newcommand \gnts{géné\-ra\-lités\xspace}

\newcommand \grl{groupe \rtl}
\newcommand \grls{groupes \rtls}

\newcommand \gRl {\hyperref[theorieGrl]{\grl}\xspace}
\newcommand \gRls {\hyperref[theorieGrl]{\grls}\xspace}

\newcommand\gtr{géné\-ra\-teur\xspace}  
\newcommand\gtrs{géné\-ra\-teurs\xspace}  


\newcommand \homo {ho\-mo\-mor\-phisme\xspace}
\newcommand \homos {ho\-mo\-mor\-phismes\xspace}

\newcommand \hmg {homo\-gène\xspace}
\newcommand \hmgs {homo\-gènes\xspace}

\newcommand \icftr {intervalle compact réticulé \ftm réel\xspace}
\newcommand \icftrs {intervalles compacts réticulés \ftm réels\xspace}

\newcommand \icl {inté\-gra\-lement clos\xspace}
\newcommand \icle {inté\-gra\-lement close\xspace}

\newcommand \icsr {intervalle compact \stm réticulé\xspace}
\newcommand \icsrs {intervalles compacts \stm réticulés\xspace}

\newcommand \icrc {intervalle compact réel clos\xspace}
\newcommand \icrcs {intervalles compact réels clos\xspace}

\newcommand \id {idéal\xspace}
\newcommand \ids {idéaux\xspace}

\newcommand \ida {\idt \agq}
\newcommand \idas {\idts \agqs}

\newcommand \idc  {\idt de Cramer\xspace}
\newcommand \idcs {\idts de Cramer\xspace}

\newcommand \idd {idéal déter\-minan\-tiel\xspace}
\newcommand \idds {idéaux déter\-minan\-tiels\xspace}

\newcommand \idema {idéal maxi\-mal\xspace}
\newcommand \idemas {idéaux maxi\-maux\xspace}

\newcommand \idep {idéal pre\-mier\xspace}
\newcommand \ideps {idéaux pre\-miers\xspace}

\newcommand \idemi {\idep minimal\xspace}
\newcommand \idemis {\ideps minimaux\xspace}

\newcommand \idf {idéal de Fitting\xspace}
\newcommand \idfs {idéaux de Fitting\xspace}

\newcommand \idif {idéal \dvlg fini\xspace}
\newcommand \idifs {idéaux \dvlgs finis\xspace}

\newcommand \idli {idéal \dvli\xspace} 
\newcommand \idlis {idéaux \dvlis\xspace} 

\newcommand \idm {idem\-potent\xspace}
\newcommand \idms {idem\-potents\xspace}
\newcommand \idme {idem\-potente\xspace}
\newcommand \idmes {idem\-potentes\xspace}

\newcommand \idp {idéal prin\-ci\-pal\xspace}
\newcommand \idps {idé\-aux prin\-ci\-paux\xspace}

\newcommand \idt {iden\-ti\-té\xspace}
\newcommand \idts {iden\-ti\-tés\xspace}

\newcommand \idtr {indé\-ter\-minée\xspace}
\newcommand \idtrs {indé\-ter\-minées\xspace}

\newcommand \ifr {idéal frac\-tion\-naire\xspace}
\newcommand \ifrs {idéaux frac\-tion\-naires\xspace}

\newcommand \imd {immé\-diat\xspace}
\newcommand \imde {immé\-diate\xspace}
\newcommand \imds {immé\-diats\xspace}
\newcommand \imdes {immé\-diates\xspace}

\newcommand \imdt {immé\-dia\-te\-ment\xspace}

\newcommand \indtr {inf-demi-treillis\xspace} 

\newcommand \inteq {intui\-ti\-vement \eqve}
\newcommand \inteqs {intui\-ti\-vement \eqves}

\newcommand \Inteq {\hyperref[defextintequiv]{\inteq}\xspace}
\newcommand \Inteqs {\hyperref[defextintequiv]{\inteqs}\xspace}

\newcommand \ing {in\-ver\-se \gne}
\newcommand \ings {in\-ver\-ses \gnes}

\newcommand \iMP {in\-ver\-se de Moo\-re-Pen\-ro\-se\xspace}
\newcommand \iMPs {in\-ver\-ses de Moo\-re-Pen\-ro\-se\xspace}

\newcommand \ipp {\idep poten\-tiel\xspace}
\newcommand \ipps {\ideps poten\-tiels\xspace}

\newcommand \ird {irré\-duc\-tible\xspace}
\newcommand \irds {irré\-duc\-tibles\xspace}

\newcommand \iso {iso\-mor\-phisme\xspace}
\newcommand \isos {iso\-mor\-phismes\xspace}

\newcommand \itf {idéal \tf}
\newcommand \itfs {idé\-aux \tf}

\newcommand \itid {intui\-ti\-vement iden\-tique\xspace}
\newcommand \itids {intui\-ti\-vement iden\-tiques\xspace}

\newcommand \iv {inversible\xspace}
\newcommand \ivs {inversibles\xspace}

\newcommand \ivdg {inverse divisoriel\xspace} 
\newcommand \ivdgs {inverses divisoriels\xspace} 

\newcommand \ivde {inverse divisorielle\xspace} 
\newcommand \ivdes {inverses divisorielles\xspace} 

\newcommand \ivda {inverse divisoriel\xspace} 
\newcommand \ivdas {inverses divisoriels\xspace} 


\newcommand \lgb {local-global\xspace}
\newcommand \lgbe {locale-globale\xspace}
\newcommand \lgbs {local-globals\xspace}

\newcommand \lin {liné\-aire\xspace}
\newcommand \lins {liné\-aires\xspace}

\newcommand \lint {liné\-ai\-rement\xspace}

\newcommand \lmo {\lot mono\-gène\xspace}
\newcommand \lmos {\lot mono\-gènes\xspace}

\newcommand \lnl {\lot \nl}
\newcommand \lnls {\lot \nls}

\newcommand \lot {loca\-lement\xspace}

\newcommand \lon {loca\-li\-sation\xspace}
\newcommand \lons {loca\-li\-sations\xspace}

\newcommand \lop {\lot prin\-cipal\xspace}
\newcommand \lops {\lot prin\-cipaux\xspace}

\newcommand \lsdz {\lot \sdz}

\newcommand \mdi {mo\-dule des \diles}

\newcommand \mlm {mo\-dule \lmo}
\newcommand \mlms {mo\-dules \lmos}

\newcommand \mlmo {ma\-tri\-ce de loca\-li\-sation
mono\-gène\xspace}
\newcommand \mlmos {ma\-tri\-ces de loca\-li\-sation
mono\-gène\xspace}

\newcommand \mlp {ma\-tri\-ce de loca\-li\-sation
prin\-ci\-pa\-le\xspace}
\newcommand \mlps {ma\-tri\-ces de loca\-li\-sation
prin\-ci\-pa\-le\xspace}

\newcommand \mo {mo\-no\"{\i}de\xspace}
\newcommand \mos {mo\-no\"{\i}des\xspace}

\newcommand \moco {\mos \com}

\newcommand \molo {morphisme de \lon\xspace}
\newcommand \molos {morphismes de \lon\xspace}

\newcommand \mom {mo\-nô\-me\xspace}
\newcommand \moms {mo\-nô\-mes\xspace}

\newcommand \moquo {morphisme de passage au quotient\xspace}
\newcommand \moquos {morphismes de passage au quotient\xspace}

\newcommand \mpf {mo\-dule \pf}
\newcommand \mpfs {mo\-dules \pf}

\newcommand \mpl {mo\-dule plat\xspace}
\newcommand \mpls {mo\-dules plats\xspace}

\newcommand \mpn {ma\-trice de \pn}
\newcommand \mpns {ma\-trices de \pn}

\newcommand \mprn {ma\-trice de \prn}
\newcommand \mprns {ma\-trices de \prn}

\newcommand \mptf {mo\-dule \ptf}
\newcommand \mptfs {mo\-dules \ptfs}

\newcommand \mrc {mo\-dule \prc}
\newcommand \mrcs {mo\-dules \prcs}

\newcommand \mtf {module \tf}
\newcommand \mtfs {modules \tf}


\newcommand \ncr{néces\-saire\xspace}       
\newcommand \ncrs{néces\-saires\xspace}       

\newcommand \ncrt{néces\-sai\-rement\xspace}       

\newcommand \ndz {régu\-lier\xspace}
\newcommand \ndzs {régu\-liers\xspace}

\newcommand \nl {simple\xspace}
\newcommand \nls {simples\xspace}

\newcommand \noco {\noe \coh}
\newcommand \nocos {\noes \cohs}

\newcommand \Noe {Noether\xspace}

\newcommand \noe {noethé\-rien\xspace}
\newcommand \noes {noethé\-riens\xspace}
\newcommand \noee {noethé\-rienne\xspace}
\newcommand \noees {noethé\-riennes\xspace}

\newcommand \noet {noethé\-ria\-nité\xspace}

\newcommand \nst {Null\-stellen\-satz\xspace}
\newcommand \nsts {Null\-stellen\-s\"atze\xspace}

\newcommand \op{opé\-ra\-tion\xspace}  
\newcommand \ops{opé\-ra\-tions\xspace}
\newcommand \opari{\op\ari}  
\newcommand \oparis{\ops\aris}  
\newcommand \oparisv{\ops\arisv}  

\newcommand \oqc {ouvert \qc}
\newcommand \oqcs {ouverts \qcs}

\newcommand \ort{or\-tho\-go\-nal\xspace}  
\newcommand \orte{or\-tho\-go\-na\-le\xspace}  
\newcommand \orts{or\-tho\-go\-naux\xspace}  
\newcommand \ortes{or\-tho\-go\-na\-les\xspace}  


\newcommand \pa {couple saturé\xspace}
\newcommand \pas {couples saturés\xspace}
 
\newcommand \paral{paral\-lèle\xspace}  
\newcommand \parals{paal\-lèles\xspace}  

\newcommand \paralm{paral\-lè\-lement\xspace}

\newcommand \pb{pro\-blè\-me\xspace}  
\newcommand \pbs{pro\-blè\-mes\xspace}  

\newcommand \peq {purement équa\-tion\-nelle\xspace}
\newcommand \peqs {purement équa\-tion\-nelles\xspace}

\newcommand \pf {de \pn finie\xspace}

\newcommand \plc {rési\-duel\-lement \zed}
\newcommand \plcs {rési\-duel\-lement \zeds}

\newcommand \Plg {Prin\-cipe \lgb}
\newcommand \plg {prin\-cipe \lgb}
\newcommand \plgs {prin\-cipes \lgbs}

\newcommand \plga {\plg abs\-trait\xspace}
\newcommand \plgas {\plgs abs\-traits\xspace}

\newcommand \Plgc {\Plg con\-cret\xspace}
\newcommand \plgc {\plg con\-cret\xspace}
\newcommand \plgcs {\plgs con\-crets\xspace}

\newcommand \pn {présen\-ta\-tion\xspace}
\newcommand \pns {présen\-ta\-tions\xspace}

\newcommand \pog {\pol \hmg\xspace}
\newcommand \pogs {\pols \hmgs\xspace}

\newcommand \Pol {Poly\-nôme\xspace}
\newcommand \Pols {Poly\-nômes\xspace}

\newcommand \pol {poly\-nôme\xspace}
\newcommand \pols {poly\-nômes\xspace}

\newcommand \poll{poly\-nomial\xspace}  
\newcommand \polls{poly\-nomiaux\xspace}  
\newcommand \polle{poly\-no\-miale\xspace}  
\newcommand \polles{poly\-no\-miales\xspace}  

\newcommand \pollt{poly\-no\-mia\-lement\xspace}  

\newcommand \polfon {\pol fon\-da\-men\-tal\xspace}
\newcommand \polmu {\pol rang\xspace}
\newcommand \polmus {\pols rang\xspace}
\newcommand \polcar {\pol carac\-té\-ris\-tique\xspace}
\newcommand \polmin {\pol mini\-mal\xspace}

\newcommand \polg {\pol g\'en\'e\-rateur\xspace}
\newcommand \polgs {\pols g\'en\'e\-rateurs\xspace}
\newcommand \polgmin {\polg minimal\xspace}

\newcommand \prc {\pro de rang constant\xspace}
\newcommand \prcs {\pros de rang constant\xspace}

\newcommand \prcc {prin\-ci\-pe de \rcc}
\newcommand \prca {prin\-ci\-pe de \rca}
\newcommand \prce {prin\-ci\-pe de \rce}

\newcommand \prmt {préci\-sé\-ment\xspace}
\newcommand \Prmt {Préci\-sé\-ment\xspace}

\newcommand \prn {pro\-jec\-tion\xspace}
\newcommand \prns {pro\-jec\-tions\xspace}

\newcommand \pro {pro\-jec\-tif\xspace}
\newcommand \pros {pro\-jec\-tifs\xspace}

\newcommand \prr {pro\-jec\-teur\xspace}
\newcommand \prrs {pro\-jec\-teurs\xspace}

\newcommand \Prt {Pro\-pri\-été\xspace}
\newcommand \Prts {Pro\-pri\-étés\xspace}
\newcommand \prt {pro\-pri\-été\xspace}
\newcommand \prts {pro\-pri\-étés\xspace}

\newcommand \ptf {\pro \tf}
\newcommand \ptfs {\pros \tf}

\newcommand \qc {quasi-compact\xspace}
\newcommand \qcs {quasi-compacts\xspace}

\newcommand \qi {qua\-si in\-tè\-gre\xspace}
\newcommand \qis {qua\-si in\-tè\-gres\xspace}

\newcommand \qnl {quasi-\nl}
\newcommand \qnls {quasi-\nls}

\newcommand \ralg {règle \agq}
\newcommand \ralgs {règles \agqs}

\newcommand \rav {racine virtuelle\xspace}
\newcommand \ravs {racines virtuelles\xspace}

\newcommand \rcc {\rcm con\-cret\xspace}
\newcommand \rca {\rcm abs\-trait\xspace}
\newcommand \rce {\rcc des é\-ga\-li\-tés\xspace}

\newcommand \rcm {recol\-lement\xspace}
\newcommand \rcms {recol\-lements\xspace}

\newcommand \rcv {recou\-vrement\xspace} 
\newcommand \rcvs {recou\-vrements\xspace}

\newcommand \rde {rela\-tion de dépen\-dance\xspace}
\newcommand \rdes {rela\-tions de dépen\-dance\xspace}

\newcommand \rdi {\rde inté\-grale\xspace}
\newcommand \rdis {\rdes inté\-grales\xspace}

\newcommand \rdl {\rde \lin}
\newcommand \rdls {\rdes \lins}

\newcommand \rdt {rési\-duel\-lement\xspace}

\newcommand \rdy {règle dyna\-mique\xspace}
\newcommand \rdys {règles dyna\-miques\xspace}

\newcommand \red {règle directe\xspace}
\newcommand \reds {règles directes\xspace}

\newcommand \rex {règle exis\-ten\-tielle simple\xspace}
\newcommand \rexs {règles exis\-ten\-tielles simples\xspace}

\newcommand \reX {\hyperref[defexistsimple]{règle exis\-ten\-tielle simple}\xspace}
\newcommand \reXs {\hyperref[defexistsimple]{règles exis\-ten\-tielles simples}\xspace}

\newcommand \rexri {règle exis\-ten\-tielle rigide\xspace}
\newcommand \rexris {règles exis\-ten\-tielles rigides\xspace}

\newcommand \rsim {règle de simplification\xspace}
\newcommand \rsims {règles de simplification\xspace}

\newcommand \rtl {réti\-culé\xspace}
\newcommand \rtls {réti\-culés\xspace}

\newcommand \rmq {\rcm de quotients\xspace} 
\newcommand \rvq {\rcv par quotients\xspace} 
\newcommand \rmqs {\rcms de quotients\xspace} %
\newcommand \rvqs {\rcvs par quotients\xspace} %

\newcommand \rpf {réduite-de-présen\-tation-finie\xspace}
\newcommand \rpfs {réduites-de-présen\-tation-finie\xspace}

\newcommand \rrl {relation de \recu \lin}
\newcommand \rrls {relations de \recu \lin}


\newcommand \sad {\salg dynamique\xspace}
\newcommand \sads {\salgs dynamiques\xspace}

\newcommand \sagq {semi\agq}
\newcommand \sagqs {semi\agqs}

\newcommand \sagc {\sagq continue\xspace}
\newcommand \sagcs {\sagqs continues\xspace}

\newcommand \salg {structure \agq}
\newcommand \salgs {structures \agqs}

\newcommand \scentrel {relation semi-implicative\xspace}
\newcommand \scentrels {relations semi-implicatives\xspace}

\newcommand \scf {schéma fini\-taire\xspace}
\newcommand \scfs {schémas fini\-taires\xspace}

\newcommand \scl {schéma \elr}
\newcommand \scls {schémas \elrs}

\newcommand \sdo {\sdr \orte}
\newcommand \sdos {\sdrs \ortes}

\newcommand \sdr {somme directe\xspace}
\newcommand \sdrs {sommes directes\xspace}

\newcommand \sdz {sans \dvz}

\newcommand \sfio {sys\-tème fondamental d'\idms ortho\-gonaux\xspace}
\newcommand \sfios {sys\-tèmes fondamentaux d'\idms ortho\-gonaux\xspace}

\newcommand \sgr {\sys \gtr}
\newcommand \sgrs {\syss \gtrs}

\newcommand \slgb {stricte\-ment \lgb}
\newcommand \slgbs {stricte\-ment \lgbs}

\newcommand \sli {\sys \lin}
\newcommand \slis {\syss \lins}

\newcommand \smq {symé\-trique\xspace}
\newcommand \smqs {symé\-triques\xspace}

\newcommand \spb {sépa\-rable\xspace}  
\newcommand \spbs {sépa\-rables\xspace}

\newcommand \spe {spéci\-fi\-cation\xspace}
\newcommand \spes {spéci\-fi\-cations\xspace}

\newcommand \spi {\spe incomplète\xspace}
\newcommand \spis {\spes incomplètes\xspace}

\newcommand \spl {sépa\-rable\xspace}  
\newcommand \spls {sépa\-rables\xspace}

\newcommand \spo {semipolynôme\xspace}
\newcommand \spos {semipolynômes\xspace}

\newcommand \spt{sépa\-ra\-bi\-lité\xspace}

\newcommand \srg {suite régu\-lière\xspace}
\newcommand \srgs {suites régu\-lières\xspace}

\newcommand \srl{suite r\'ecur\-rente \lin}
\newcommand \srls{suites r\'ecur\-rentes \lins}
\newcommand \Srls{Suites r\'ecur\-rentes \lins}

\newcommand \stf {strictement fini\xspace}
\newcommand \stfs {strictement finis\xspace}
\newcommand \stfe {strictement finie\xspace}
\newcommand \stfes {strictement finies\xspace}

\newcommand \stl {stablement libre\xspace}
\newcommand \stls {stablement libres\xspace}

\newcommand \stm {strictement\xspace}

\newcommand \str {\stm réticulé\xspace}
\newcommand \stre {\stm réticulée\xspace}
\newcommand \strs {\stm réticulés\xspace}
\newcommand \stres {\stm réticulées\xspace}

\newcommand \sul {supplé\-men\-taire\xspace}
\newcommand \suls {supplé\-men\-taires\xspace}

\newcommand \Sut {Support\xspace}
\newcommand \Suts {Supports\xspace}
\newcommand \sut {support\xspace}

\newcommand \syc {\sys de coordon\-nées\xspace}
\newcommand \sycs {\syss de coordon\-nées\xspace}

\newcommand \syp {\sys \poll}
\newcommand \syps {\syss \polls}

\newcommand \sys {sys\-tème\xspace}
\newcommand \syss {sys\-tèmes\xspace}

\newcommand \talg {théorie \agq}
\newcommand \talgs {théories \agqs}

\newcommand \tco {théorie cohé\-rente\xspace}
\newcommand \tcos {théories cohé\-rentes\xspace}

\newcommand \tdy {théorie dyna\-mique\xspace}
\newcommand \tdys {théories dyna\-miques\xspace}

\newcommand \tel {théorie exis\-ten\-tielle\xspace}
\newcommand \tels {théories exis\-ten\-tielles\xspace}

\newcommand \telri {théorie exis\-ten\-tielle rigide\xspace}
\newcommand \telris {théories exis\-ten\-tielles rigides\xspace}

\newcommand \tf {de type fini\xspace}

\newcommand \tfo {théorie formelle\xspace}
\newcommand \tfos {théorie formelles\xspace}

\newcommand \tgm {théorie \gmq}
\newcommand \tgms {théories \gmqs}

\newcommand \Tho {Théo\-rème\xspace}
\newcommand \Thos {Théo\-rèmes\xspace}
\newcommand \tho {théo\-rème\xspace}
\newcommand \thos {théo\-rèmes\xspace}

\newcommand \thoc {théo\-rème$\mathbf{^*}$~}

\newcommand \tpe {théorie \peq}
\newcommand \tpes {théories \peqs}

\newcommand \trdi {treil\-lis dis\-tri\-bu\-tif\xspace}
\newcommand \trdis {treil\-lis dis\-tri\-bu\-tifs\xspace}

\newcommand \trel {trans\-for\-mation \elr}
\newcommand \trels {trans\-for\-mations \elrs}

\newcommand \umd {unimo\-du\-laire\xspace}
\newcommand \umds {unimo\-du\-laires\xspace}

\newcommand \unt {uni\-taire\xspace}
\newcommand \unts {uni\-taires\xspace}

\newcommand \uvl {uni\-versel\xspace}
\newcommand \uvle {uni\-ver\-selle\xspace}
\newcommand \uvls {uni\-versels\xspace}
\newcommand \uvles {uni\-ver\-selles\xspace}


\newcommand \vfn {véri\-fi\-cation\xspace}
\newcommand \vfns {véri\-fi\-cations\xspace}

\newcommand \vmd {vec\-teur \umd}
\newcommand \vmds {vec\-teurs \umds}

\newcommand \zed {z\'{e}ro-di\-men\-sionnel\xspace}
\newcommand \zede {z\'{e}ro-di\-men\-sion\-nelle\xspace}
\newcommand \zeds {z\'{e}ro-di\-men\-sion\-nels\xspace}
\newcommand \zedes {z\'{e}ro-di\-men\-sion\-nelles\xspace}

\newcommand \zedr {\zed réduit\xspace}
\newcommand \zedrs {\zeds réduits\xspace}

\newcommand \zmt {\tho de Zariski-Grothen\-dieck\xspace}


\newcommand \cof {cons\-truc\-tif\xspace}
\newcommand \cofs {cons\-truc\-tifs\xspace}

\newcommand \cov {cons\-truc\-tive\xspace}
\newcommand \covs {cons\-truc\-tives\xspace}

\newcommand \coma {\maths\covs}
\newcommand \clama {\maths clas\-siques\xspace}

\renewcommand \cot {cons\-truc\-ti\-vement\xspace}

\newcommand \matn {mathé\-ma\-ticien\xspace}
\newcommand \matne {mathé\-ma\-ti\-cienne\xspace}
\newcommand \matns {mathé\-ma\-ticiens\xspace}
\newcommand \matnes {mathé\-ma\-ti\-ciennes\xspace}

\newcommand \maths {mathé\-ma\-tiques\xspace}
\newcommand \mathe {mathé\-ma\-tique\xspace}

\newcommand \prco {démons\-tration \cov}
\newcommand \prcos {démons\-trations \covs}



\theoremstyle{plain}
\newtheorem{ftheorem}{Théorème}[section]
\newtheorem{fthdef}[ftheorem]{Théorème et définition}
\newtheorem{fpstf}[ftheorem]{Positivstellensatz formel}
\newtheorem{fpst}[ftheorem]{Positivstellensatz}
\newtheorem{flemma}[ftheorem]{Lemme} 
\newtheorem{fcorollary}[ftheorem]{Corolaire} 
\newtheorem{fconjecture}[ftheorem]{Conjecture} 
\newtheorem{fproposition}[ftheorem]{Proposition} 
\newtheorem{fpbu}[ftheorem]{Problème universel} 
\newtheorem{fprpta}[ftheorem]{Propriétés attendues} 
\newtheorem{fpropdef}[ftheorem]{Proposition et définition}
\newtheorem{ffact}[ftheorem]{Fait} 
\newtheorem{fplcc}[ftheorem]{Principe local-global concret}

\theoremstyle{definition}
\newtheorem{fconvention}[ftheorem]{Convention}
\newtheorem{fdefinition}[ftheorem]{Définition} 
\newtheorem{fdfni}[ftheorem]{Définition informelle} 
\newtheorem{fdefinitions}[ftheorem]{Définitions} 
\newtheorem{fnotation}[ftheorem]{Notation} 
\newtheorem{fproblem}[ftheorem]{Problème}
\newtheorem{fquestion}[ftheorem]{Question}
\newtheorem{fquestions}[ftheorem]{Questions}
\newtheorem{fcontext}[ftheorem]{Contexte}
\newtheorem{fdefinitionc}[ftheorem]{Définition\etoz} 
\newtheorem{fdefinota}[ftheorem]{Définition et notation}

\theoremstyle{remark}
\newtheorem{fexample}[ftheorem]{Exemple}
\newtheorem{fexamples}[ftheorem]{Exemples}
\newtheorem{fnotes}[ftheorem]{Notes}
\newtheorem{fremark}[ftheorem]{Remarque}
\newtheorem{fremarks}[ftheorem]{Remarques}
\newtheorem{fcomment}[ftheorem]{Commentaire}

\stMF

\maketitle

Cet article est paru en français dans: \textsl{Communications in Algebra}, {\bf 42}: 3768–3781, (2014). 

DOI: 10.1080/00927872.2013.794360

\begin{abstract} Soit $\V$ un domaine de valuation.
Nous donnons un \algo  pour calculer une base du $\V$-saturé d'un sous-\mtf
d'un \Vmo libre (avec une base éventuellement infinie). Nous l'appliquons pour calculer le $\V$-saturé d'un sous-$\VX$-\mtf
de $\VX^n$ ($n\in\NN^*$). Ceci permet enfin de calculer
un \sgr fini pour les syzygies sur  $\VX$ d'une famille finie de vecteurs de $\VX^k$.
\end{abstract}

\sni {\small {\bf Mots clés:}  Saturation, Cohérence, Syzygies, Anneau de valuation, Calcul formel, Algèbre constructive.}

\selectlanguage{english}
\begin{abstract}
\smallskip We give an algorithm for computing the $\V$-saturation of any finitely
generated submodule of $\V[X]^n$ ($n \in \mathbb{N}^*$),
where $\V$ is a valuation domain. This allows us to compute a finite
system of generators for the syzygy module of any finitely generated
submodule of $\V[X]^k$.
\end{abstract}

\sni {\small\textbf{Key words:}  Saturation, Coherence, Syzygies, Valuation domains, Computer algebra, Constructive Algebra.}

\selectlanguage{french}

\small

\setcounter{tocdepth}{4}
\markboth{Table des matières}{Table des matières}

\printcontents[french]{}{1}{}
\normalsize

\pagestyle{myheadings} \markboth{Un algorithme pour le calcul des syzygies sur  $\mathbf{V}[X]$}{
H. Lombardi, C. Quitté et I. Yengui}


\section*{Introduction} \label{fsec Introduction}
\addcontentsline{toc}{section}{Introduction}

Cela fait partie du folklore (voir par exemple \cite[Th.\ 7.3.3]{fGlaz}) que pour un domaine de valuation~$\V$, l'anneau $\V[X]$ est cohérent, (i.e., le module des syzygies pour un \itf de $\VX$ est \tf). 
Néanmoins la preuve dans la référence citée découle d'un résultat profond et difficile dans le gros article \cite{fGR}. 
Et il semble en outre qu'il n'existe pas d'\algo dans la littérature existante pour ce résultat remarquable.

Dans le cas des anneaux \noes \cohs (non \ncrt des \ddvs),
on sait que l'anneau de \pols $\R[X_1,\dots,X_n]$ est \egmt \noe \coh. 
Une preuve constructive se trouve dans \cite{fric74}
et est exposée dans le livre \citealt*{fMRR}. 
On peut aussi dans ce cas utiliser les bases de Gr\"obner, introduites
par Buchberger pour les anneaux de \pols sur des corps 
(voir par exemple \cite{fLou,fHY,fY}).

Néanmoins, la \noet n'est pas le fin mot de l'affaire puisque
le résultat concernant la cohérence dans le cas des corps s'étend facilement aux anneaux \zedrs
(aussi appelés Von Neuman réguliers, ou absolument plats).

Dans \citealt*{fLSY}, un \algo est donné pour le calcul d'une base de Gr\"obner
pour un \itf de $\VX$ dans le cas  d'un \ddv de dimension 1
(non \ncrt \noe).
On peut en déduire un \algo pour le calcul d'un  \sgr fini pour les syzygies d'un \itf de~$\VX$.

Rappelons (voir par exemple \citealt*{fMRR}) que sur un anneau \coh tout \mpf $M$ est un module \coh
(i.e., le module des syzygies pour un sous-\mtf de $M$ est \tf).

Dans l'article présent, nous donnons un \algo débarrassé de toute hypothèse \noee, ainsi que de toute hypothèse concernant la dimension de Krull, 
pour  le calcul
d'un \sgr fini des syzygies d'une famille 
finie de vecteurs de $\VX^k$.
Nous pensons fournir ainsi
par la m\^eme occasion la première preuve constructive du résultat
(car notre \algo est prouvé constructivement).

Rappelons d'abord que pour un sous-\Rmo $M$ d'un module $N$ avec $\R$
un anneau intègre, le \textsl{$\R$-saturé de $M$ dans $N$} est le \Rmo 
$$\Sat_{\R,N}(M)=\sotq{x\in N}{\exists a\in \R^*,\,ax\in M}.$$
S'il y a plusieurs anneaux en présence, on dira pour préciser le ``$\R$-saturé de $M$ dans $N$''.
 Dans le cas où $N$ 
est un module libre (de la forme $\R^n $, $n\in \NN$ ou $\R^{(I)}$, $I$ infini)
on obtient par extension des scalaires un \Kev $\K\otimes N\simeq\K^n $ ou $\K^{(I)}$, où $\K$ est le corps des fractions de $\R$. On a alors aussi  $\Sat_{\R,N}(M)=\K.M\cap N$,
où  $\K.M$ est le sous-\Kev de $\K\otimes N$  engendré par $M$.
Le module $M$ est dit  $\R$-saturé s'il est égal à son $\R$-saturé.
En \gnl, le $\R$-saturé d'un $\RX$-\mtf dans $N=\RX^n $ n'a aucune raison d'\^etre lui-m\^eme un  $\RX$-\mtf,
mais c'est ce qui arrive si l'anneau $\R$ est  un \ddv $\V$.
 
Nous donnons dans la section \ref{fsatVmotf} un \algo incrémental pour calculer une base du $\V$-saturé 
d'un sous-\mtf d'un \Vmo libre (avec une base éventuellement infinie).
Cet \algo n'est pas un scoop, mais il est mis dans une forme où 
nous sommes capables de l'utiliser, dans la section \ref{fsatVXmotf}, pour calculer un \sgr fini du $\VX$-module obtenu en $\V$-saturant un sous-$\VX$-\mtf
de $\VX^n$ ($n\in \NN^*$). 

Ceci démontre que le $\V$-saturé d'un $\VX$-\mtf
dans $\VX^n$ est bien un $\VX$-\mtf.

Enfin on obtient comme corolaire immédiat,
dans la section \ref{fVXsysygies}, le calcul
d'un \sgr fini pour les syzygies sur  $\VX$ d'une famille 
finie de vecteurs de $\VX^k$.

\medskip
Dans cet article tous les anneaux sont commutatifs et unitaires.

\section{Saturation d'un \Vmo \tf dans un \Vmo libre} \label{fsatVmotf}

\noindent  {\bf Terminologie:} 
Nous utiliserons dans cet article 
la terminologie usuelle de l'algèbre constructive comme
dans \citealt*{fMRR}, bien adaptée au Calcul Formel.

\smallskip Pour un anneau $\R$ arbitraire on note $\R\eti$ le groupe multiplicatif des unités de $\R$.
L'anneau~$\R$ est dit \textsl{discret} lorsque l'on a un \algo qui 
décide si $x=0$ ou $x\neq 0$ pour un \elt arbitraire de $\R$.
Un anneau $\R$ est dit \textsl{local} lorsque l'on a de fa{\c c}on explicite l'implication
\[
\forall x,y\in\R,\,x+y\in\R\eti\; \Longrightarrow\;(x \in\R\eti\;\vu \;y\in\R\eti)
\]
Il revient au m\^eme de demander
\[
\forall x\in\R,(x \in\R\eti\;\vu \;1+x\in\R\eti)
\]
Un anneau local non nul $\R$ admet pour unique idéal maximal son \textsl{radical de Jacobson} 
\[
\Rad(\R)=\sotq{x\in\R}{1+x\R\subseteq \R\eti}
\]
(pour un anneau arbitraire $\R$ cet idéal est égal à l'intersection des idéaux maximaux lorsqu'on se situe en mathématiques classiques).

L'anneau quotient $\k=\R/\Rad(\R)$ est un corps, appelé \textsl{corps résiduel} de $\R$.
L'anneau local $\R$ est dit \textsl{\dcd} lorsque l'on a de fa{\c c}on explicite la disjonction 
\[\forall x\in\R,\,(x\in\R\eti\;\vu\;x\in\Rad(\R)).
\] 
Dans ce cas le corps résiduel est un corps discret. On a alors un \algo qui décide la disjonction  ``$x=0$ ou $x$ inversible'' pour tout $x\in\k$.

\medskip Soit $I$ un ensemble, fini ou infini, muni d'une relation d'ordre total ``discrète'', i.e. nous disposons d'un \algo qui décide la disjonction
$$i<j \quad \vu\quad  i=j \quad \vu\quad  i>j$$ 
pour $i,j\in I$. On peut alors, pour n'importe quel anneau $\R$, considérer le module libre $\R^{(I)}$ dans lequel on notera $(\e_i)_{i\in I}$ la base naturelle. Si $\R$ est un anneau discret tout
vecteur $a=\sum_{i\in J} a_i\e_i$ dans~$\R^{(I)}$
($J$ est une partie finie de $I$) peut \^etre testé nul ou non nul, et lorsqu'il est non nul, on peut déterminer le plus petit indice
$i$ pour lequel $a_i\neq 0$.

Pour déterminer un sous-\Rmo saturé d'un module libre $\R^{(I)}$
on fera appel aux  lemmes~\ref{flemSat1} et~\ref{flemSat2} suivants.

\begin{flemma} \label{flemSat1}
Soit $\R$ un anneau intègre et $N$ un \Rmo sans torsion. Si $N=M\oplus P$
alors~$M$ est saturé dans $N$.
\end{flemma}

Dans toute la suite de cette section,  on suppose que $\R$ est un anneau local intègre \dcd avec $\k=\R/\Rad(\R)$.
\begin{fdefinota} \label{fnotaPiv}~
\begin{enumerate}
\item Un vecteur $C=\sum_i c_i \e_i$ de $\R^{(I)}$ 
 (que nous voyons comme un vecteur colonne)
 est dit \textsl{primitif} s'il est \rdt non nul, i.e., l'un de ses \coes
 est une unité. Dans ce cas-là nous notons $\piv(C)$ le plus petit indice
 $i$ pour lequel $c_i$ est \rdt non nul. C'est \textsl{l'indice pivot de $C$}. Nous notons $\cq(C)$ 
 (\textsl{coefficient pivot de $C$}) le scalaire $c_i$ correspondant.
\item Une famille finie $(C^k)_{k\in K}$ de vecteurs primitifs est dite \textsl{$\k$-échelonnée}
si les $\piv(C^k)$ sont deux à deux distincts. Une famille $\k$-échelonnée est aussi appelée simplement  \textsl{échelonnée}.
On note alors 
  $$
  \piv(C)=\sotq{\piv(C^k)}{k\in K}.
  $$
Une matrice à \coes dans $\R$ sera dite échelonnée si ses vecteurs colonnes sont primitifs et forment une famille échelonnée.
\item La famille $(C^k)_{k\in K}$ est dite \textsl{en forme $\k$-échelonnée
stricte}, lorsque $K$ est un ensemble totalement ordonné, si elle est échelonnée et si en outre, pour $j<k$ dans $K$ le \coe d'indice $\piv(C^j)$ de $C^k$ est nul.
\end{enumerate}
\end{fdefinota}

 
\begin{flemma} \label{flemSat2}\label{flemEchStrict}
Si une famille finie $C=(C^k)_{k\in K}$ de vecteurs de $\R^{(I)}$ 
est $\k$-échelonnée, elle
forme une base du \Rmo $M$ qu'elle engendre et $M$ admet 
pour \sul le \Rmo libre
 $$
 P= \bigoplus\nolimits_{j\in J}\R\,\e_j\quad \hbox{avec } J=\sotq{j\in I}{j\notin \piv(C)}.
 $$
En outre  l'ensemble $\piv(C)$ ne dépend que du module  $M$. En effet un indice $j$ est dans $\piv(C)$ \ssi il existe un vecteur primitif $U$
dans $M$ avec $\piv(U)=j$.  
\end{flemma}
NB: les deux \Rmos libres en question sont automatiquement $\R$-saturés 
d'après le lemme~\ref{flemSat1}.
\begin{proof}
Nous donnons l'argument lorsque $I$ est fini, mais il s'adapte
facilement au cas \gnl.
Si on ordonne la famille formée des $\e_j$ pour $j\in J$ et des $C^{k}$ pour $k\in K$, en ordre d'indices pivots croissants. 
Alors la matrice formée par ces colonnes est \rdt triangulaire avec des \coes inversibles sur la diagonale, donc elle est \rdt inversible, donc elle est inversible. Ceci montre que $M$ et $P$ sont \suls et admettent les bases
voulues. Le reste est laissé au lecteur.
\end{proof}

\subsection*{L'algorithme de saturation}

\begin{fcontext} \label{fcontext1}
 
\noindent Soit $\V$ un \ddv, i.e. un anneau intègre dans lequel pour tous $a,b$
on a $a\mid b$ ou $b\mid a$, i.e. plus \prmt on a un \algo
qui décide (pour $a,b$ donnés dans $\V$) la disjonction
$$
\exists x\in\V,\,a=xb\quad \vu\quad \exists x\in\V,\,b=xa
$$
et fournit l'\elt $x$. 
On sait que $\V$ est un anneau local (supposons $a+b$ inversible, si $a$ divise $b$ il divise $a+b$ et il est inversible, si $b$ divise $a$ il divise $a+b$ et il est inversible). 
On note $\K$ son corps de fractions et $\k$ son corps résiduel.
Comme $\V$ est supposé intègre de manière explicite, le corps $\K$
est un corps discret. 
On suppose en outre que $\V$ est \textsl{\dcd}, ce qui signifie que l'on a un \algo 
pour décider si un \elt de $\V$
est une unité. En particulier les lemmes \ref{flemSat1} et \ref{flemSat2} sont satisfaits avec
l'anneau $\V$.
\end{fcontext}

\noindent \textsl{Remarque.}
Dans un domaine de valuation \dcd, on a un test pour répondre à la question
 \gui{$a\mid b$?}. En effet, pour $a,b$ non nuls, si $a=bx$, alors $a\mid b$ \ssi $x$ est une unité.

\medskip  
On considère un sous-\Vmo $M=\V \,a^1+\cdots+\V \, a^m$ de $\V^{(I)}$.
L'objectif de cette section est de donner un \algo pour calculer une 
base 
du $\V$-saturé de $M$ dans $\V^{(I)}$, \Vmo que nous notons de manière abrégée
$\Sat(M)$.
En fait, comme seul un nombre fini d'indices sont en cause, on peut aussi 
bien supposer que $I$ est fini et que $M$ est engendré par les colonnes d'une matrice $F$.
Pour visualiser la chose nous pouvons écrire les lignes de la matrice
en ordre décroissant pour les indices.
 
Une manière brutale de calculer $\Sat(M)$ serait de réduire $F$ 
à la forme de Smith par des manipulations \elrs.
On voit alors qu'après un changement de base convenable, le module $M$
est engendré par des vecteurs $v_if_i$ (où les $f_i$ forment une partie
d'une base
et les $v_i$ sont non nuls). Dans ces conditions, la module $\Sat(M)$
est simplement le module engendré par ces $f_i$.
Ceci nous indique que $\Sat(M)$ est un \Vmo libre ayant pour \sul un autre 
\Vmo libre.

\smallskip
 En fait nous préférons procéder de manière moins brutale et
obtenir une base de  $\Sat(M)$ comme les vecteurs colonnes d'une matrice
 $\k$-échelonnée  $G$ que nous calculons à partir de $F$ au moyen d'opérations très simples.

Une première opération, de réduction d'un vecteur, que nous noterons 
$\RedPrim$ consiste à remplacer un vecteur  $a$, supposé non nul, par
$a/u$, où $u$ est un pgcd de ses \coes, par exemple $u$ est le \coe d'indice minimum parmi ceux qui divisent tous les autres, auquel cas le \coe pivot du vecteur réduit est égal à $1$.

Le traitement que nous faisons subir à la matrice $F$ 
pour la ramener à une forme $\k$-échelonnée stricte $G$
telle que $\Sat(\Im(F))= \Im(G)$  est la
suivante. Elle procède en traitant une après l'autre les colonnes de la matrice initiale.
Notez qu'au départ, la matrice vide est en forme $\k$-échelonnée stricte. 

Supposons qu'on ait traité quelques colonnes initiales de la matrice et qu'on
ait obtenu une matrice $\k$-échelonnée stricte avec pour colonnes 
$C^1,\dots,C^r$.

On veut traiter une nouvelle colonne, que l'on appelle $C=\sum_i c_i\e_i$.
On procède comme suit.
\begin{enumerate}
\item Pivot de Gauss:  on opère des manipulations \elrs
de colonnes classiques 
  $$
  C\leftarrow C- \frac{c_s}{c_{j,s}}\,C^j,
  $$ 
ici $s=\piv(C^j)$
et $c_{j,s}=\cq(C^j)$.
Cette opération est faite successivement avec les colonnes $C^1,\dots,C^r$.
On obtient alors une colonne $C'$. 
\item Si $C'=0$, on ne la rajoute pas. 
La matrice reste  en forme $\k$-échelonnée stricte.
Le \Kev engendré par les colonnes $C^1,\dots,C^r,C$ admet $(C^1,\dots,C^r)$ pour base.
\item Si $C'\neq 0$, on remplace $C'$ par sa forme réduite primitive
$C''=\RedPrim(C')$. Nous rajoutons
alors $C''$ comme dernière colonne $C^{r+1}$ de la matrice.
Et la nouvelle matrice est  en forme $\k$-échelonnée stricte.
Le \Kev engendré par les colonnes $C^1,\dots,C^r,C$ admet 
$(C^1,\dots,C^r,C'')$ pour base.   
\end{enumerate}

Par construction $\Im(G)$ est contenu dans $\Sat(\Im(F))$ et le \Vmo $\Im(G)$ est saturé parce que~$G$ est en  forme $\k$-échelonnée stricte.
En fait, aussi bien dans le cas 2. que dans le cas 3., on voit par \recu que l'on a bien construit une base du $\V$-saturé engendré par les première colonnes (jusqu'à la colonne $C$). \`A la fin du processus on a donc $\Sat(\Im(G))=\Sat(\Im(F))$.  
Notre \algo remplit bien le but fixé.

\begin{ftheorem} \label{fthAlgoSat}~

\noindent L'\algo de saturation décrit ci-dessus calcule, à partir
d'une matrice $F$ à \coes dans $\V$ une matrice $G$  en  forme $\k$-échelonnée stricte telle que $\Im(G)=\Sat(\Im(F))$.
\\
Cet \algo est \gui{incrémental} au sens suivant. Si l'on traite une matrice
$[\,F_1\mid F_2\,]$, on obtient une matrice $[\,G_1\mid G_2\,]$ où $G_1$
est la matrice obtenue en traitant la matrice $F_1$.
\end{ftheorem}

\medskip 
Le lemme suivant nous sera utile dans la prochaine section.

\begin{flemma} \label{flemPivG}
Dans la procédure ``Pivot de Gauss'' décrite ci-dessus, si $C$ est primitif et si l'indice~$\piv(C)$ est distinct des  $\piv(C^j)$
pour $j=1,\dots,r$, alors le vecteur $C'$ obtenu est primitif avec $\piv(C)=\piv(C')$.  
\end{flemma}
%
\begin{proof}{}
On considère l'affectation $C\leftarrow C- \frac{c_s}{c_{j,s}} C^j$ où $s=\piv(C^j)$
et $c_s$ est le \coe de $C$ sur la ligne $s$.
Posons $\ell=\piv(C)$. Le \coe $c_\ell$ sur la ligne $\ell$ de $C$ est  remplacé par $c_\ell - \frac{c_s}{c_{j,s}}\cdot c_{j,\ell}$, où $c_{j,\ell}$ est le \coe  sur la ligne $\ell $ de $C^j$.
Si $\ell >s$, $c_s$
est \rdt nul, si $\ell <s$ c'est $c_{j,\ell}$ qui est \rdt nul, dans les deux cas le \coe $c_\ell$ reste \rdt
inchangé.  
\end{proof}
%

\section{La $\V$-saturation d'un $\VX$-\mtf} \label{fsatVXmotf}

Le travail que nous faisons dans cette section est un peu plus délicat
et semble, étrangement, tout à fait nouveau. Il réalise 
en Calcul Formel un résultat
théorique simple apparemment nouveau, et qui à fortiori n'avait jusqu'à maintenant aucune preuve constructive.

\begin{ftheorem} \label{fthSat} On se situe toujours dans le contexte \ref{fcontext1}.
Si $M$ est un sous-$\VX$-\mtf de~$\VX^n$ alors le $\V$-saturé de 
$M$ dans  $\VX^n$ est \egmt un $\VX$-\mtf.  
\end{ftheorem}

\noindent {\it Remarque.} Notons qu'en mathématiques classiques, tout \ddv satisfait les hypothèses du contexte \ref{fcontext1}, par application du principe du tiers exclu.
Notre preuve constructive du \tho~\ref{fthSat} fournit donc aussi
une preuve en mathématiques classiques sous la seule hypothèse que $\V$
est un \ddv. La m\^eme remarque s'applique pour tous les résultats de cet article.

\medskip La \dem du \tho résulte de la correction 
de l'\algo qui calcule un \sgr fini du  $\V$-saturé.

Vu comme $\VX$-module on a la base naturelle de   $\VX^n$
notée $(\f_1,\dots,\f_n)$.
Nous nous intéressons alors à une base naturelle de  $\VX^n$ comme \Vmo, qui est formée par les $\e_{i,k}=X^k\,\f_i$ avec l'ensemble d'indice $I=\intervalle{1..n}\times \NN$.
Nous munissons $I$ de l'ordre lexicographique pour lequel  
 $$
X^h\, \f_i < X^k\,\f_j \,\hbox{ si }\, i<j \;\hbox{ ou } \;i=j \hbox{ et } h<k.
 $$

Lorsque le module $\VX^{n}$ est vu comme un \Vmo avec la base naturelle des $X^k\,\f_j$, nous parlons  des \gui{\coos} sur cette base.
Lorsqu'il est vu comme un \VXmo  avec la base naturelle des~$\f_j$,
nous parlons  des \gui{\coes} sur cette base.

On dispose au départ d'une liste $S=[s^1,\dots,s^m]$ de vecteurs dans 
 $\VX^n$ qui forment un \sgr de $M$. 
 On suppose \spdg que $m\geq 1$ et que les $s^k$ sont non nuls.
On note 
$$ 
E=\V\,s^1+\cdots+\V\,s^m,\quad E_j=X^j\,E,\quad F'_k=\som_{j=0}^kE_j \quad\hbox{ et }\quad G'_k=\Sat_{\V,\VX^n}(F'_k).
$$

On peut décrire $F'_k$ et $G'_k$ comme les modules images de deux matrices
$F_k$ et $G_k$. La matrice $F_k$ est donnée, on la traite au moyen de l'\algo
de saturation de la section précédente, ce qui donne la matrice~$G_k$. 

La question qui se pose est de certifier qu'à partir d'un certain $k$,
rien ne sert de continuer, car les \elts rajoutés dans la base de $G'_k$ 
 laissent inchangé le $\VX$-module engendré (notez que le \Vmo $G'_k=\Im(G_k)$ grandit à chaque étape car $E\neq 0$).
 
\medskip Nous avons besoin de préciser quelques notations. 
Nous appelons \gui{degré de $E$} et nous notons~$d$ le plus grand degré d'une des coordonnées de l'un des $s^{k}$. 
De la m\^eme manière, $d+k$ sera le degré de~$E_k$ ou celui de~$F_k$. La matrice $F_k$ peut donc \^etre vue comme une matrice
avec $n(1+d+k)$ lignes et $m(1+k)$ colonnes.
 
Si $a$ est un vecteur $\V$-primitif de $\VX^n$, et si $\piv(a)=(j,r)\in I$ nous notons 
$$
\idx(a) := j \hbox{ et } \PrimMon(a) :=r.
$$
L'entier $\idx(a)$ est appelé l'\textsl{index de $a$},  l'entier $\PrimMon(a)$
son \textsl{premier exposant résiduel} et le couple  $\piv(a)$ est \textsl{l'indice pivot de $a$}. Tout couple $(j,r)\in I$ sert d'indice pour un vecteur  $X^{r}\,\f_{j}$ de la $\V$-base naturelle de $\VX^n$. 

Notons $H_k$ la matrice formée des colonnes que l'on rajoute à $G_{k-1}$
pour obtenir la matrice $G_k$.

\begin{ffact} \label{ffactHksuffit}
Pour calculer la matrice $G_{k+1}$ à partir de la matrice $G_k$, au lieu de traiter les \gtrs de $E_{k+1}$ (i.e. la liste $X^{k+1}S$), on peut se contenter de traiter les
vecteurs colonnes de $XH_k$.  
\end{ffact}
%
\begin{proof}{} 
Considérons la procédure simplifiée décrite ci-dessus. \\
Notons $\ov{G_k}$ les matrices successives obtenues par cette procédure simplifiée.
\\
On vérifie sans difficulté par \recu sur $k$ que le \Kev engendré par les colonnes
de~$\ov{G_k}$ est le sous-espace $\K F'_k$ de $\KX^{m}$.
En effet, toute colonne réduite à $0$ est dans le \Kev engendré par les colonnes précédentes. Et toute colonne non réduite à $0$, engendre, modulo les colonnes précédentes, une fois réduite, le m\^eme \Kev 
que la colonne qui lui donne naissance.
\\
Les colonnes de $\ov{G_k}$ forment une base d'un
sous-\Vmo saturé de $\VX^{m}$, qui est donc égal à $\K F'_k\cap\VX^{m}$. Ceci montre que $\ov{G_k}=G_k$
\end{proof}

Dans la suite, nous faisons référence à la procédure simplifiée, mais nous notons $H_k$ et $G_k$ au lieu de $\ov{H_k}$ et $\ov{G_k}$.

Aux  matrices $H_k$ et $G_k$ nous associons plusieurs entiers: 
\begin{itemize}
\item L'entier $r_k$ est le nombre de colonnes de $G_k$, autrement dit le rang du \Vmo libre $\Im(G_k)$.
\item L'entier $n_k$, \textsl{nombre d'index pivots présents dans $H_k$}, est le cardinal de l'ensemble des $i\in\intervalle{1..n}$ tel qu'il existe une colonne $C$ de $H_k$ avec $\idx(C)=i$.
D'après le lemme \ref{flemEchStrict}, tous les index pivots présents
dans $H_{k-1}$ sont présents dans $H_k$, d'où l'on déduit par \recu
que  $n_k$ est aussi le nombre d'index pivots présents dans $G_k$. Donc la suite $n_k$ est une suite croissante.
\item L'entier $u_k$, \textsl{nombre de \coos disponibles pour $G_k$ au vu de $n_k$}, est égal à $n_k(1+d+k)$. Si $n_{k+1}=n_k$ 
on a $u_{k+1}=u_k+n_k$.
\item L'entier $\delta_k$, \textsl{défaut de $H_k$}, est le nombre de colonnes $C$
de $H_k$ telles qu'il existe une autre colonne~$C'$ de $H_k$ 
avec $\idx(C)=\idx(C')$
et $\PrimMon(C)<\PrimMon(C')$. Une telle colonne $C$ sera dite \textsl{surnuméraire}.  On a donc $r_k\leq u_k$ et $r_k=r_{k-1}+n_k+\delta_k$.
\item L'entier $\Delta_k=u_{k+1}-r_k$ est la \textsl{place disponible à occuper à l'étape $k+1$
si $n_k=n_{k+1}$}.\\ Si $n_k=n_{k+1}$, on a $\Delta_k=u_{k+1}-r_k=(u_k+n_k)-(r_{k-1}+n_k+\delta_k)=u_k-r_{k-1}-\delta_k=\Delta_{k-1}-\delta_k$. 
\end{itemize}

\medskip 
Pour visualiser le défaut $\delta_k$ de $H_k$ et la manière dont il évolue
lorsque $k$ augmente nous aurons recours, après la preuve, à 
des figures illustrant ce qui peut se passer.
Mais peut-\^etre la lecture de la preuve sera-t-elle facilité si on lit d'abord les commentaires qui accompagnent ces figures.

Le point essentiel est le suivant. 

D'après le lemme  \ref{flemPivG}, on est certain que lorsque l'on va traiter
les colonnes successives de $H_{k+1}=XH_k$ au moyen de $G_k$, 
tout pivot  $(j,r)$ d'une colonne de $H_k$ se retrouvera, décalé d'un cran,
i.e. en position $(j,r+1)$, comme pivot 
d'une colonne  de $H_{k+1}$, sauf dans le cas où il s'agit d'un indice
$(j,r+1)$  déjà présent dans $G_0$.
Dans ce dernier cas, ou bien la collision réduit à $0$ la colonne
de $XH_k$ (ce qui fait diminuer le défaut), ou bien un nouveau pivot est occupé par la colonne réduite
(et rendue primitive). Ce nouveau pivot peut avoir deux effets distincts.
Ou bien il se produit sur un index déjà occupé, et ne fait pas diminuer le défaut. Ou bien il se produit sur un index inoccupé, dans ce cas,  
le défaut diminue de 1 et le nombre $n_k$ augmente entre  $n_k$ et $n_{k+1}$. 
On a donc établi la première affirmation du lemme suivant.

\begin{flemma} \label{flemFinDeLalgoCertaine}
La suite $\delta_k$ est décroissante au sens large. Elle aboutit certainement à $0$ pour $k$ assez grand. 
\end{flemma}
%
\begin{proof}{}
On a déjà remarqué que si $n_{k+1}=n_k$ et $\delta_k>0$ alors $\Delta_{k+1}<\Delta_k$. Pour un $k$ assez grand on obtient donc  $\delta_k=0$
ou $n_{k+1}>n_k$. Dans le second cas, on reproduit la situation précédente.
Comme la suite $n_k$ est bornée par $n$, cela ne peut se produire qu'un nombre fini de fois.
\end{proof}
%

\begin{flemma} \label{flemFinDeLalgoCorrect}
Si $\delta_k=0$ le $\VX$-module engendré par $G_k$ est le $\V$-saturé du
$\VX$-module engendré par les $s^j$ donnés au départ. On peut donc arr\^eter l'\algo. 
\end{flemma}
%
\begin{proof}{}
Puisque la suite $\delta_k$ reste désormais nulle, il suffit de prouver que
les colonnes de~$H_{k+1}$ sont dans le $\VX$-module engendré par les
colonnes de $G_k$. Or vu le lemme \ref{flemPivG}, les colonnes de~$H_{k+1}$
sont dans le \Vmo $\Im(G_k)+X\Im(H_k)$.
\end{proof}

\medskip \noindent {\bf Un exemple avec des figures.}

\smallskip 
La figure 1 représente les indices pivots de $G_0=H_0$. 
Les $6$ cercles blancs  sont les \elts de $\piv(H_0)$. 

Les cercles ou carrés noirs pleins correspondent
à des  \elts de la $\V$-base où aucun indice pivot de $G_0$ n'est présent.
Les carrés noirs sont mis pour les index de pivots non présents: si toute une ligne est noire, on met des carrés pour insister.
\\
Dans le cas présent on a donc $n=5$, $d=4$, $n_0=4$, $r_0=6$, $u_0=20$,
$\Delta_0=14$, $\delta_0=2$.

On a entouré d'un grand cercle les deux \elts $\piv(H_0)$ des colonnes surnuméraires, qui correspondent à $\delta_0=2$.

\begin{figure}[ht]   
\begin{center}
\includegraphics*[width=9cm]{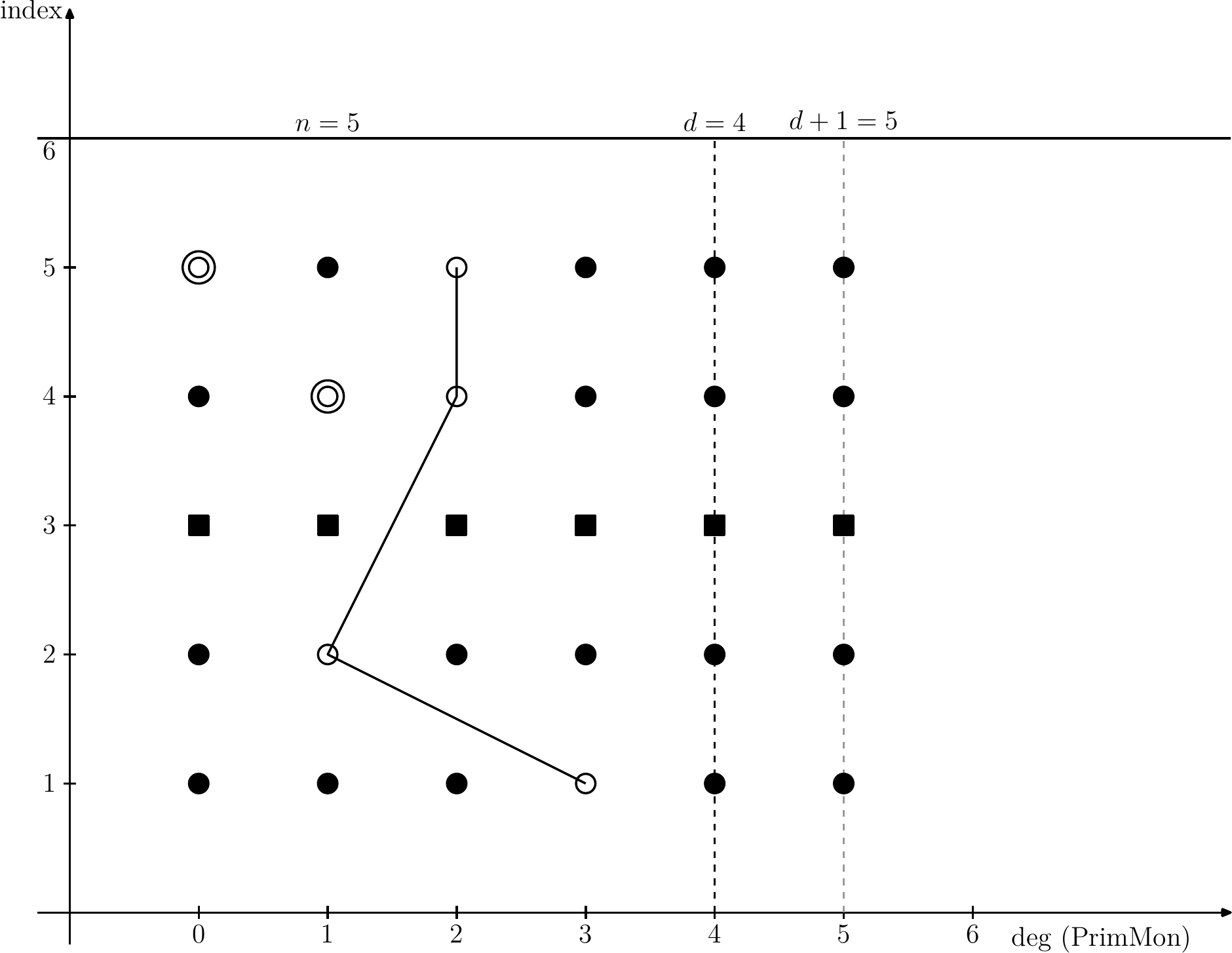}

\caption[figure 1]
{\label{ffig1} }

\end{center}
\end{figure}

La ligne brisée noire joint les
$\piv(C)$ pour les colonnes non surnuméraires de $H_0$. \\
D'après le lemme  \ref{flemPivG}, on est certain que lorsque l'on va traiter
les colonnes successives de $H_1=XH_0$ au moyen de $G_0$, 
tout pivot  $(j,r)$ d'une colonne de $H_0$ se retrouvera, décalé d'un cran,
i.e. en position $(j,r+1)$ comme pivot 
d'une colonne  de $H_1$, sauf dans le cas où il s'agit d'un indice
$(j,r+1)$  déjà présent dans $G_0$.

Dans le cas de la figure 1, tous les pivots de $H_0$ à l'exception du pivot $(4,1)$ se retrouvent décalés d'un cran dans $H_1$. En particulier le pivot $(5,1)$ apparaitra dans $H_1$,
et l'on voit que ce sera
pour une colonne surnuméraire (donc $\delta_1\geq 1$).
\\
Lorsqu'on va traiter la
colonne $XC$ telle que $\piv(C)=(4,1)$, une \textsl{collision} va se produire: 
un pivot de Gauss va \^etre effectué
pour réduire à $0$ le \coe en position $(4,2)$ de $XC$
et  la procédure de saturation va produire, ou bien le vecteur nul
(auquel cas $\delta_1=1$ et $H_1$ n'aura que $5$ colonnes, $r_1=11$), ou bien un vecteur $C''$ tel que $\piv(C'')$ remplisse une case
non occupée dans l'espace à priori disponible (vecteurs de degré $\leq 5$)
avec \ncrt  $\piv(C'')\notin \piv(G_0)$.
Dans ce cas on aura $r_1=12$. 


Nous allons maintenant examiner trois possibilités pour ce $\piv(C'')$, et nous donnons les 3 figures correspondantes pour $H_1$.
Les pivots de $H_0$ seront des cercles (vides) gris et ceux de $H_1$ des cercles (vides) noirs. 
Les carrés ou ronds noirs pleins correspondent 
de nouveauà des cases vides qui pourraient à priori
\^etre remplies à l'avenir.

Dans le cas de la figure 2 on a  $n_1=4$, $\delta_1=2$, $u_1=24$, 
$\Delta_1=12$.
Lorsque l'on traitera $XH_1$ au moyen de $G_1$,
des pivots en positions $(1,5)$, $(2,4)$, $(4,4)$ et $(5,4)$
seront produits dans $H_2$. Et deux collisions, respectivement en $(2,1)$ et $(5,2)$, donneront des résultats plus difficiles à prévoir.
On pourra avoir $r_2=16$ avec $\delta_2=0$, ou  $r_2=17$ (avec $\delta_2=1$ si $n_2=4$, ou $\delta_2=0$ si $n_2=5$),
ou encore  $r_2=18$.

\begin{figure}[ht]   
\begin{center}
\includegraphics*[width=9cm]{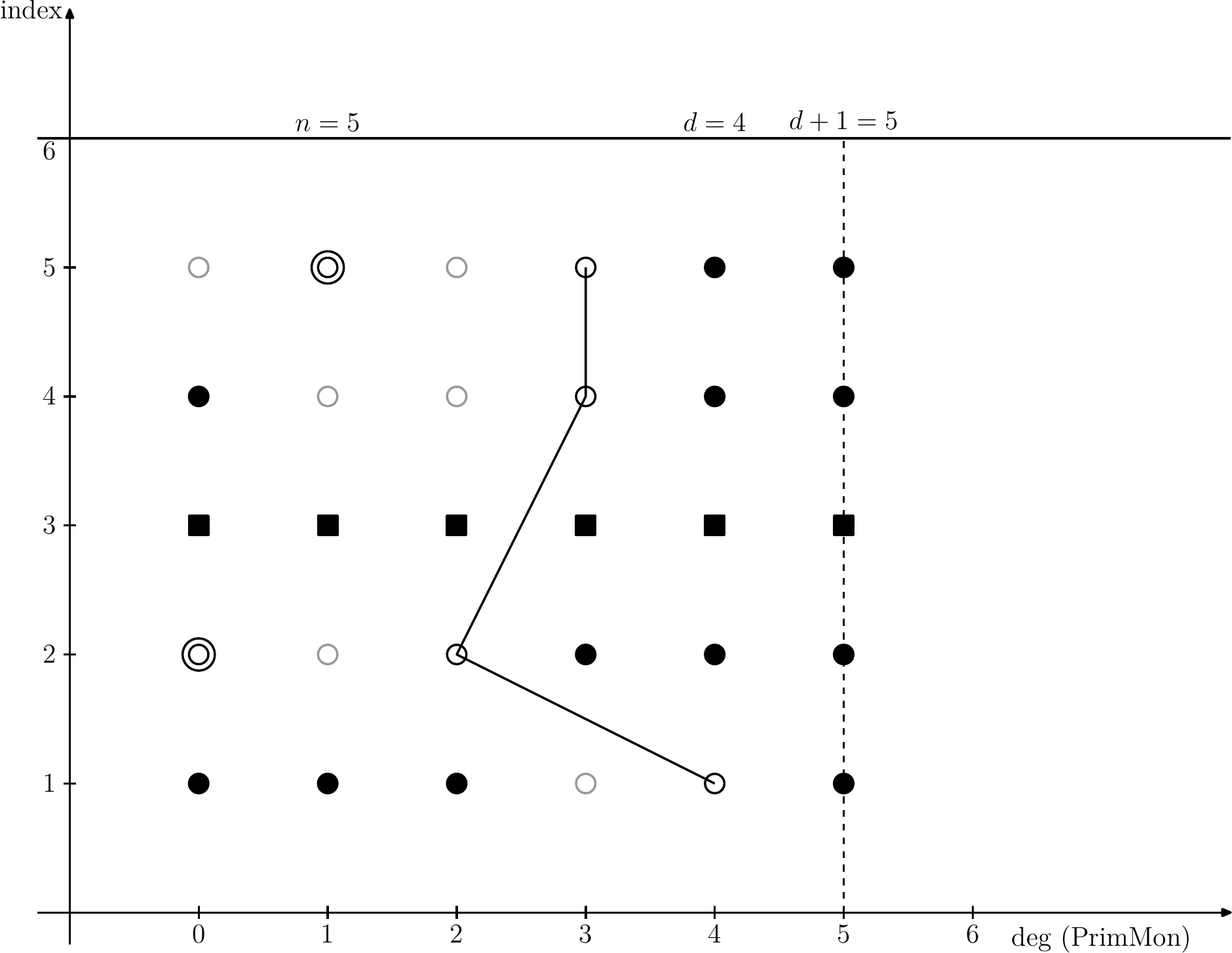}

\caption[figure 2]
{\label{ffig2} Si la collision en $(4,2)$ produit un pivot en position $(2,0)$ }

\end{center}
\end{figure}
	
\begin{figure}[ht]   
\begin{center}
\includegraphics*[width=9cm]{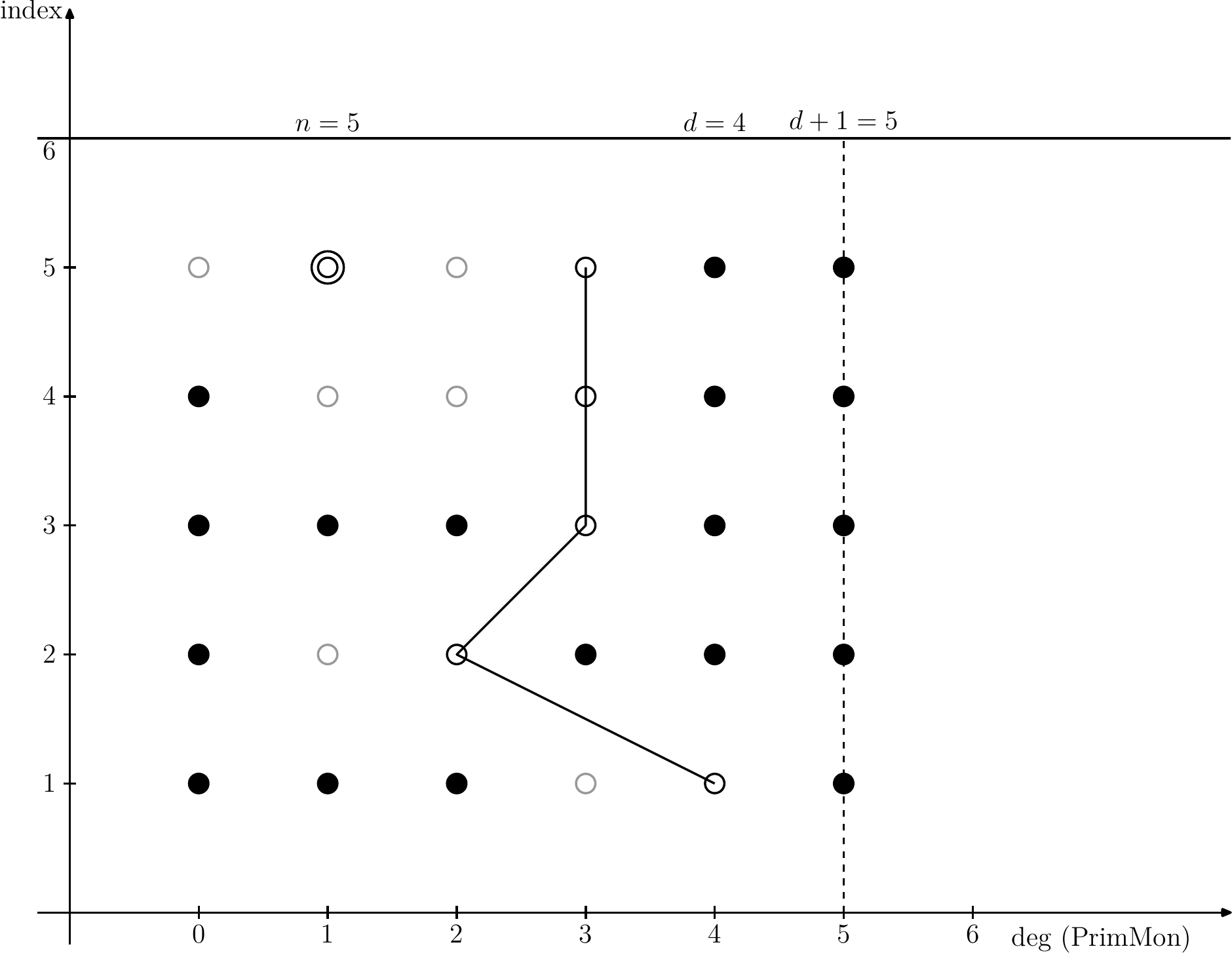}

\caption[figure 3]
{\label{ffig3} Si la collision en $(4,2)$ produit un pivot en $(3,3)$}

\end{center}
\end{figure}


Dans le cas de la figure 3  on a  $n_1=5$, $\delta_1=1$, $u_1=30$, $\Delta_1=18$.
Lorsque l'on traitera $XH_1$ au moyen de $G_1$,
des pivots en positions $(1,5)$, $(2,3)$, $(3,4)$, $(4,4)$ et $(5,4)$
seront produits dans $H_2$. Et une collision, en $(5,2)$, donnera un résultat plus difficile à prévoir.
Cela pourra réduire le vecteur à $0$, auquel cas $\delta_2=0$, ou produire un nouveau
vecteur, auquel cas $\delta_2=1$, car tous les index sont maintenant occupés par des pivots.  Le nouveau vecteur aura à priori pour pivot n'importe lequel des carrés noirs indiqués sur la figure, ou aussi
un pivot de degré $6$.

\begin{figure}[ht]   
\begin{center}
\includegraphics*[width=9cm]{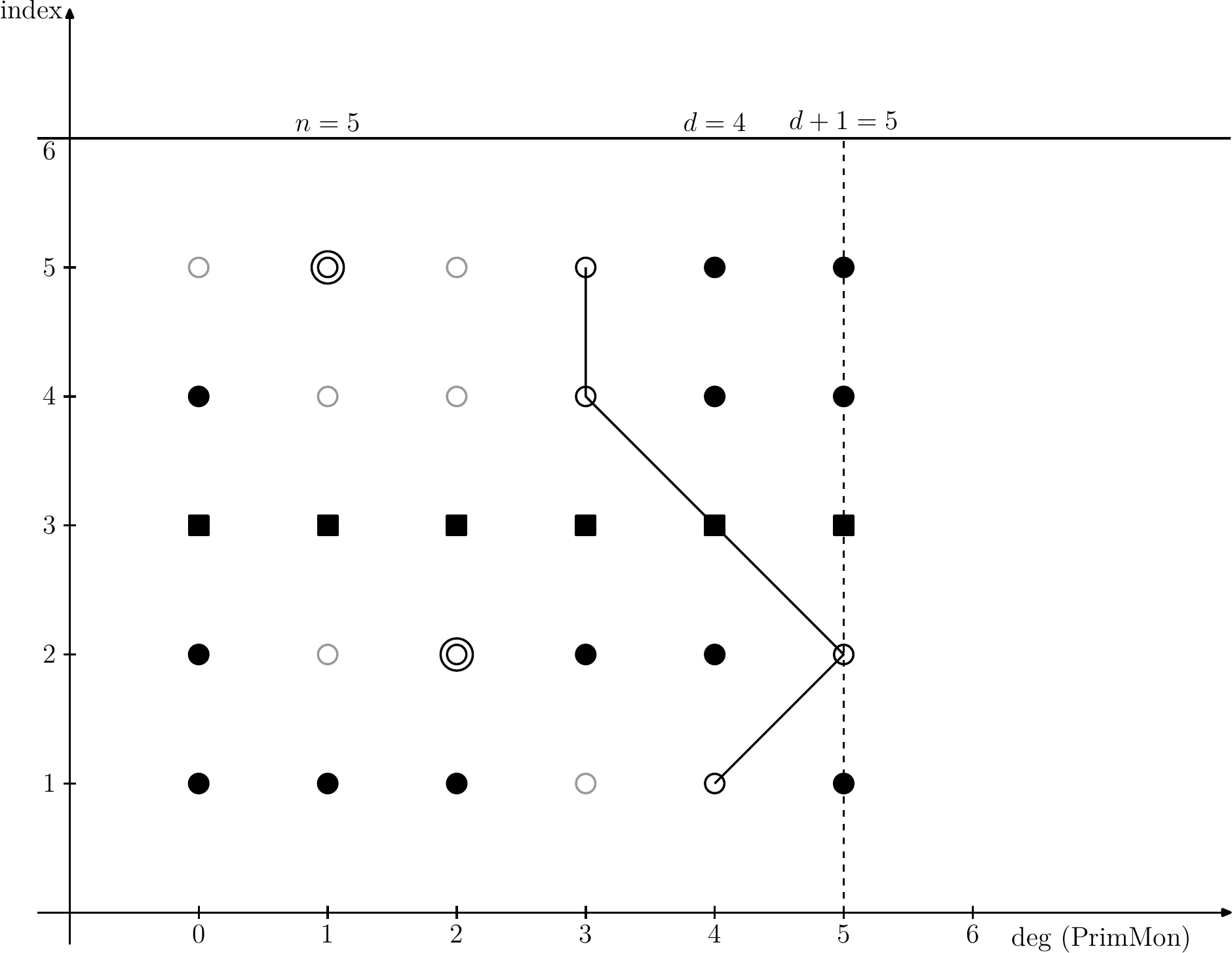}

\caption[figure 4]
{\label{ffig4} Si la collision en $(4,2)$ produit un pivot en $(2,5)$}

\end{center}
\end{figure}


Dans le cas de la figure 4 on a  $n_1=4$ et $\delta_1=2$.
Lorsque l'on traitera $XH_1$ au moyen de $G_1$,
des pivots en positions $(1,5)$, $(2,4)$, $(4,4)$ et $(5,4)$
seront produits dans $H_2$. Et une collision, en $(5,2)$, donnera un résultat plus difficile à prévoir.
La colonne surnuméraire de pivot $(2,2)$ ne produira à priori pas de collision,
sauf dans le cas où la collision certaine, précédemment évoquée, est traitée avant 
et  donne un vecteur réduit de pivot $(2,3)$.

%

\section{Module des syzygies pour un $\VX$-\mtf} \label{fVXsysygies}

\begin{ftheorem} \label{fthsyzygies} 
On se situe toujours dans le contexte \ref{fcontext1}.
Soit $u_1,\dots,u_n\in\VX^k$ et $s^1,\dots,s^m\in\KX^n$ des \gtrs du
module des syzygies pour $(u_1,\dots,u_n)$ sur $\KX$. On peut supposer que
les $s^j$ sont dans $\VX^n$. Alors le  module des syzygies pour $(u_1,\dots,u_n)$ sur $\VX$ est égal au $\V$-saturé  dans $\VX^n$ du $\VX$-module engendré par 
$s^1,\dots,s^m$. En conséquence vu le \tho \ref{fthSat}
ce module est \tf. 

\noindent En particulier $\VX$ est un anneau \coh.
\end{ftheorem}

%
\begin{proof}
Un vecteur $f=(f_1,\dots,f_n)\in\VX^n$ tel que $\sum_jf_ju_j=0$
s'écrit sous forme 
  $$
  f=a_1s^1+\cdots+a_ms^m
  $$ 
avec des $a_j\in\KX$. 
En multipliant cette relation par un $\alpha\neq 0$ convenable dans $\V$, 
on obtient les~$\alpha\, a_i\in\VX$. Ceci montre que $\alpha\,f$ est dans le $\VX$-module engendré par les $s^j$. Et donc $f$ est dans le
$\V$-saturé du $\VX$-module engendré par les $s^j$.
La réciproque est évidente.
\end{proof}

Ceci conduit à l'\algo suivant pour calculer un \sgr fini du module des syzygies pour $u_1,\dots,u_n\in\VX^k$ sur $\VX$:
\begin{enumerate}
\item Calculer des vecteurs  $v^1,\dots,v^m\in\KX^n$ qui forment un \sgr fini du  module des syzygies pour $(u_1,\dots,u_n)$ sur $\KX$.
\item En multipliant chaque $v^j$ par un $\alpha_j\in\K$ convenable, le remplacer par un $s^j\in\VX^n$ primitif.
\item Calculer au moyen de l'\algo de la section \ref{fsatVXmotf} un \sgr fini
du $\V$-saturé du $\VX$-module  $\gen{s^1,\dots,s^m}$ dans $\VX^n$.
\end{enumerate}


\section{Annexe: des codes Magma} \label{fcode}

%
%
%
%

Nous présentons dans cette annexe des codes magma pour calculer
le $\V$-saturé d'un sous-\VXmo de type fini d'un module $\VX^{n}$
 (si $\V$ est un \ddv \dcd) en suivant la méthode décrite dans
 l'article.
 
Les commentaires sont ou bien sur une ligne, précédés de {\tt //},
ou bien sur plusieurs lignes, entre les signes {\tt /*} et {\tt */}.
{\small 
\begin{verbatim}
%!TEX encoding =  UTF-8 Unicode

EchelonStrict := function(L, v0)
/*   L est une séquence strictement échelonnée de vecteurs (primitifs) de V[X]^n, 
     v0 un vecteur de V[X]^n non nul
     On note V.L le sous-V-module de V[X]^n engendré par L.
     Il est saturé dans V[X]^n  car la séquence L est strictement échelonnée.
     La fonction retourne un vecteur v et un booléen.
     Le booléen est faux si v est dans V.L + V.v0, vrai sinon
     Si v0 est dans V.L, alors v = 0
     Sinon, v est primitif, la séquence L' = (L, v) est strictement échelonnée
     et on a Sat(V.L + V.v0) = V.L'
*/

v := v0 ;
for w in L do 
   index, M, a := PivotAndCoefficient(w) ;
     // tuer le coefficient de v en position (index, M)
   v := v - a^(-1)*MonomialCoefficient(v[index], M)*w ;
    // si v est nul, c'est que v0 est dans L.V
   if v eq 0 then return v, false ; end if ;
end for ;

  // Ici, v est non nul. On le rend primitif
c := Contenu(v) ;   v := QuotientExact(v, c) ;
return v, not IsUnit(c) ;

end function ;

/*
S = s^1, ... , s^m est une séquence dans  V[X]^n
On considère dans V[X]^n le sous-V-module engendré par S , X*S , ... , X^k*S
et on note G'_k son saturé (comme V-module) (avec G'_k = {0} pour k < 0)
Dans la suite on calcule une V-base  G  de  G'_k  et  une liste  B  dans  V[X]^n
Cette liste  B,  extraite de  G,  suffira à engendrer, comme  V[X]-module,
le  V-saturé du V[X]-module engendré par  S.  
*/
/* 
Initialisations avec  G'_0
On calcule une V-base strictement échelonnée de G'_0
et le début de la liste  B  qui correspond à  G'_0
On initialise aussi  N,   qui est le nombre de vecteurs de  S  
qui restent en vie après cette première réduction 
*/

G := [Universe(S)| ] ; B:= G ;
// G est une liste vide d'objets du type de S
for v in S do 
   w , new := EchelonStrict(G,v) ; 
   if w ne 0 then Append(~G,w) ; end if ;
end for ;
B := G ; N := #G ;

/*  valeurs successives de G : V-base strictement échelonnée du V-module G'_k  
    valeurs successives de B : G allégée de vecteurs inutiles du point de vue 
                            du V[X]-moule engendré par  G'_k
    valeurs successives de N : dim_V (G'_k/G'_{k-1}) = nb de cols de H_k =
                            nombre de vecteurs de  X^k*S  qui restent en vie
                            après le calcul de G
*/

while true do

  // H <-- les N derniers vecteurs de G (= V-base de G'_{k-1}/G'_{k-2})
  H := G[#G-N+1 .. #G] ;

  // On calcule le défaut et on sort s'il est nul
  defaut := Defaut(H) ;

  if defaut eq 0 then break ; end if ;
  H := [X*v : v in H] ;
  N := 0 ;
  for v in H do
    w, new := EchelonStrict(G,v) ;
    if w ne 0 then Append(~G, w) ; N := N+1 ; end if ;
    if new then Append(~B, w) ; end if ;
  end for ;

end while ;
\end{verbatim}
}

\addcontentsline{toc}{section}{Références}

\markboth{Références}{Références}

\small
~
\bibliographystyle{plainnat-fr}

\normalsize

\end{document}